\documentclass[a4paper]{amsart}

\usepackage[alphabetic]{amsrefs}
\usepackage{color,hyperref,pgfplots,newtxtext,newtxmath}
\pgfplotsset{compat=1.18}

\theoremstyle{plain}
\newtheorem{theorem}{Theorem}[section]
\newtheorem{lemma}[theorem]{Lemma}

\newtheorem{coro}[theorem]{Corollary}

\theoremstyle{definition}

\theoremstyle{remark}
\newtheorem{remark}[theorem]{Remark}

\numberwithin{equation}{section}
\newcommand{\abs}[1]{\lvert#1\rvert}
\newcommand{\labs}[1]{\left\lvert\,#1\,\right\rvert}

\newcommand{\Lr}[1]{\left(#1\right)}
\newcommand{\lr}[1]{\Bigl(#1\Bigr)}
\newcommand{\set}[2]{\left\{\,#1\,\mid\,#2\,\right\}}

\newcommand{\nm}[2]{\|\,#1\,\|_{#2}}
\newcommand{\jump}[1]{[\![#1]\!]}
\newcommand{\ave}[1]{\{\!\{#1\}\!\}}
\newcommand{\wnm}[1]{|\!|\!|#1|\!|\!|_{\iota,h}}
\newcommand{\inm}[1]{|\!|\!|#1|\!|\!|_{\iota}}

\def\adj{L^{\text{adj}}}

\newcommand{\mc}[1]{\mathcal{#1}}
\newcommand{\mb}[1]{\mathbb{#1}}

\def\al{\alpha}

\def\del{\delta}

\def\na{\nabla}
\def\eps{\epsilon}
\def\pa{\partial}
\def\Om{\Omega}

\def\lam{\lambda}
\def\x{\times}

\def\md{\mathrm{d}}

\def\C{\mathbb C}
\def\D{\mathbb D}
\def\R{\mathbb R}
\def\T{\mc{T}}

\def\dx{\,\mathrm{d}\boldsymbol{x}}

\def\dsx{\,\md\sigma(\boldsymbol{x})}

\DeclareMathOperator{\divop}{div}

\def\negint{{\int\negthickspace\negthickspace\negthickspace
\negthinspace -}}
\newcommand{\nn}{\nonumber}
\begin{document}
\title[A broken Hardy inequality on FE and application for SGE]{A broken Hardy inequality on finite element space and application to strain gradient elasticity}
\author[Y. L. Liao]{Yulei Liao}
\address{Department of Mathematics, Faculty of Science, National University of Singapore, 10 Lower Kent Ridge Road, Singapore 119076, Singapore}
\email{ylliao@nus.edu.sg}
\author[P. B. Ming]{Pingbing Ming}
\address{LSEC, Institute of Computational Mathematics and Scientific/Engineering Computing, AMSS, Chinese Academy of Sciences, Beijing 100190, China}
\address{School of Mathematical Sciences, University of Chinese Academy of Sciences, Beijing 100049, China}
\email{mpb@lsec.cc.ac.cn}
\thanks{This work was funded by National Natural Science Foundation of China through Grant No. 12371438.}
\begin{abstract}
We illustrate a broken Hardy inequality on discontinuous finite element spaces, blowing up with a logarithmic factor with respect to the meshes size. This is motivated by numerical analysis for the strain gradient elasticity with natural boundary conditions. A mixed finite element pair is employed to solve this model with nearly incompressible materials. This pair is quasi-stable with a logarithmic factor, which is not significant in the approximation error, and converges robustly in the incompressible limit and uniformly in the microscopic material parameter. Numerical results back up that the theoretical predictions are nearly optimal. Moreover, the regularity estimates for the model over a smooth domain have been proved with the aid of the Agmon-Douglis-Nirenberg theory.
\end{abstract}
\keywords{Hardy inequality; Mixed finite elements; Nearly incompressible strain gradient elasticity; Uniform error estimate}
\date{\today}
\subjclass[2020]{Primary 65N15, 65N30; Secondary 26D10, 46E35, 74K20}
\maketitle
\section{Introduction}
The Hardy inequality for integrals is shown in~\cite{Hardy:1920} that
\[
    \nm{f/x}{L^p((0,\infty))}\le\dfrac{p}{p-1}\nm{f'}{L^p((0,\infty))}
\]
for all $p>1$ and $f\in W_0^{1,p}((0,\infty))$. The constant $p/(p-1)$ is best~\cite{Landau:1926}. It has been extended to dimension $d\ge 2$ and $p\neq d$ in~\cite{Hardy:1952inequalities,Evans:2015} that
\[
    \nm{f/\rho}{L^p(\mb{R}^d)}\le\dfrac{p}{\abs{p-d}}\nm{\nabla f}{L^p(\mb R^d)}
\]
for all $f\in W_0^{1,p}(\mb R^d)$ when $1\le p<d$, and for $f\in W_0^{1,p}(\mb R^d\backslash\{\boldsymbol{0}\})$ if $p>d$. Here the radius $\rho(\boldsymbol{x})=\abs{\boldsymbol{x}}$. Again, the constant $p/\abs{p-d}$ is best.
\iffalse Types of Hardy inequalities are used for the study of elliptic equations~\cite{Kalf:1972,Brezis:1983,Brezis:1997,Gkikas:2009}, heat equations~\cite{Peral:1995,Vazquez:2000}, spectra for operators~\cite{Herbst:1977,Egorov:1990,Birman:1991}, Sobolev embedding~\cite{Triebel:1999,EdmundsTriebel:1999}, uncertainty principles~\cite{Lieb:1976,Fefferman:1983} and other fields.\fi
However, as to the critical case $p=d$, it is shown in~\cite[\S~1.2.5]{Evans:2015} that the Hardy inequality is invalid for $f\in W_0^{1,d}(\mb R^d\backslash\{\boldsymbol{0}\})$. There are two ways to establish the critical Hardy inequality. One approach is to shrink the function space as shown in~\cite[Proposition 2.2]{Nazarov:2006}: There exists a positive constant $C$ depending on $d$ such that
\[
    \nm{f/\rho}{L^d(\mb R^d)}\le C\nm{\nabla f}{L^d(\mb R^d)}
\]
for $f\in W_0^{1,d}(\mb R^d\backslash\ell)$, where $\ell$ is an arbitrary ray with vertex at $\boldsymbol{0}$. Another approach is to introduce a logarithmic factor in the weighted function. \iffalse It is observed in~\cite{Solomyak:1994} that when $p=d=2$, there exists a constant $C>0$ such that
\[
    \nm{f/(\rho(1+\abs{\log\rho}))}{L^2(\mb R^2)}\le C\nm{\nabla f}{L^2(\mb R^2)}
\]
for all $f\in H_0^1(\mb{R}^2\backslash\{\boldsymbol{0}\})$ satisfying
\[
    \int_{\rho=1}f(\boldsymbol{x})\mathrm{d}\sigma(\boldsymbol{x})=0.
\]
The logarithmic factor here is only necessary for the radial function part of the spherical harmonic series expansion of $f$. Also, the logarithmic factor can be dropped if the function satisfies
\[
\int_{\rho=r}f(\boldsymbol{x})\mathrm{d}\sigma(\boldsymbol{x})=0\qquad\text{for all}\quad r>0.
\]\fi
It was proved  in~\cite[Theorem 2.8]{EdmundsTriebel:1999} that there exists a constant $C>0$ such that
\[
\nm{f/(\rho(1+\abs{\log\rho}))}{L^d(\mb R^d)}\le C\nm{\nabla f}{L^d(\mb R^d)}
\]
for all $f\in W_0^{1,d}(\mb{R}^d)$. Unfortunately, the dilation invariance of the original Hardy inequality is absent from %under $f_\lambda(\boldsymbol{x}){:}=f(\lambda\boldsymbol{x})$, 
the above inequality.

In this work, we focus on the critical case $p=d$ on a bounded domain. Let $B_1^d$ be the unit ball in $\mb R^d$. It follows from the above discussion that the Hardy inequality holds on $B_1^d$ only if $f=0$ on an ray $\ell$ with vertex $\boldsymbol{0}$.  In~\cite[Lemma 17.4]{Tartar:2007introduction} and~\cite[Corollary 6]{Machihara:2013}, they show for $p=d=2$ that
\[
    \nm{f/(\rho\log\rho)}{L^2(B_1^2)}\le 2\nm{\nabla f}{L^2(B_1^2)}
\]
for $f\in H_0^1(B_1^2)$, and the constant $2$ is best. Later, the authors in~\cite[Theorem 1.1]{Ioku:2015} generalized the above inequality to $p=d\ge 2$ as
\[
\nm{f/(\rho\log\rho)}{L^d(B_1^d)}\le\dfrac{d}{d-1}\nm{\nabla f}{L^d(B_1^d)}
\]
for $f\in W_0^{1,d}(B_1^d)$. Here the constant $d/(d-1)$ is best and it is scaling invariant~\cite[proposition A.2]{Ioku:2015} in the sense that $f_\lambda(\boldsymbol x){:}=f(\rho^{\lambda-1}\boldsymbol x)$. Though there are broad extensions of Hardy inequalities in~\cite{Ghoussoub:2013functional,Tertikas:2018}, there is no critical Hardy inequality that is valid for all $f\in W_0^{1,d}(B_1^d\backslash\{\boldsymbol{0}\})$ without logarithmic factor in the weighted function. 

We plan to prove a new Hardy inequality for a continuous finite element function satisfying $f(\boldsymbol{0})=0$, i.e., 
there exists $C$ independent of $h$ such that
\begin{equation}\label{eq:hardy1}
    \nm{f/\rho}{L^d(B_1^d)}\le C\log(1/h)\nm{\nabla f}{L^d(B_1^d)}. 
\end{equation}
Hence, the Hardy inequality for $p=d$ holds with a logarithmic factor in terms of the mesh size $h$. Meanwhile this factor seems optimal with respect to $h$, at least in $d=1$. \iffalse Similar results are not rare in finite element spaces. A typical representation is the Sobolev embedding that in ordinary $W^{1,d}(B_1^d)\not\hookrightarrow L^\infty(B_1^d)$ if $d\ge 2$, while the embedding operator may be bounded with a logarithmic factor for piecewise polynomial functions~\cite{Brenner:2004Sobolev}.Specifically, let $\mc T_h(B_1^d\backslash\{\boldsymbol{0}\})$ be a simplex triangulation of  $B_1^d\backslash\{\boldsymbol{0}\}$ with maximum meshes size $h$. We assume all element in $\mc T_h$ are shape-regular in the sense of Ciarlet and Raviart~\cite{Ciarlet:2002} and satisfying the inverse assumption. 
for all continuous finite element function $f$ on $\mc T_h(B_1^d\backslash\{\boldsymbol{0}\})$ with $f(\boldsymbol{0})=0$. \fi 
This inequality is a direct consequence of Theorem~\ref{thm:Hardy}, in which we have proved an analogous Hardy inequality for the discontinuous finite element functions. It is worthwhile to mention that the domain $B_1^d\backslash\{\boldsymbol{0}\}$ here may be easy to extend to a bounded cone with its vertex at $\boldsymbol{0}$ or a start-shaped domain.

The motivation of this broken Hardy inequality is to analyze the stability of the finite element approximation of the nearly incompressible strain gradient elasticity with natural boundary conditions. Let $\Omega\subset\mb R^d$ be a bounded domain. The basic tool for analyzing the linear elasticity and the strain gradient elasticity with essential boundary condition is 
\[
\divop[H_0^m(\Omega)]^d=L_0^2(\Omega)\cap H_0^{m-1}(\Omega)
\]
for any positive integer $m$ and the divergence operator has a bounded right inverse when $\Omega$ is Lipschitz~\cite{Bog:1979,Costabel:2010,Galdi:2011}.
%the inversion of the divergence operator is bounded for %the Lipschitz domain %$\Omega$~\cite{Bog:1979,Costabel:2010,Galdi:2011}. 
As to the strain gradient elasticity with natural boundary condition, an analogous identity is 
\[
\divop [H_0^1(\Omega)\cap H^2(\Omega)]^d=L_0^2(\Omega)\cap H^1(\Omega).
\] 
Unfortunately, the boundedness of the inversion of the divergence operator requires that $\partial\Omega$ is at least $C^2$~\cite{Danchin:2013}. Specially, for the polygon $\Omega\subset\mb{R}^2$, the space of $\divop[H_0^1(\Omega)\cap H^2(\Omega)]^2$ is measured by $H^1$-norm plus a weighted semi-norm with the weight that corresponds to the distance to each vertex of $\Omega$~\cite{Arnold:1988}. Exploiting the fact that this semi-norm can be viewed as the left hand side of the broken Hardy inequality~\eqref{eq:hardy1}, we establish the stability of a finite element pair approximation of the strain gradient elasticity with natural boundary condition posed on a convex polygon. This opens up possibilities for more practical and efficient numerical simulations in scenarios where the domain lacks smoothness.

In what follows we review the strain gradient elasticity model. Let $\Omega\subset\mb{R}^2$ be a bounded polygon and $\boldsymbol{u}:\Omega\to\mb R^2$ solves
\begin{equation}\label{eq:sgbvp}
\left\{
\begin{aligned}
(\iota^2\Delta-I)\Lr{\mu\Delta\boldsymbol{u}+(\lambda+\mu)\na\divop\boldsymbol{u}}&=\boldsymbol{f}\quad&&\text{in\;}\Om,\\
\boldsymbol{u}=\partial_{\boldsymbol{n}}\boldsymbol{\sigma n}&=\boldsymbol{0}\quad&&\text{on\;}\pa\Om,
\end{aligned}\right.
\end{equation}
where $\mu$ and $\lam$ are Lam\'e constants, and $\iota$ is the material parameter satisfying $0<\iota\le 1$. The stress $\boldsymbol{\sigma}{:}=\C\boldsymbol{\epsilon}(\boldsymbol{u})$ with the strain $\boldsymbol{\epsilon}(\boldsymbol{u}){:}=(\nabla\boldsymbol{u}+\nabla\boldsymbol{u}^T)/2$ and
\(
\C_{ijkl}=\lam\del_{ij}\del_{kl}+2\mu\del_{ik}\del_{jl},
\)
where $\delta_{ij}$ is the Kronecker delta function. Problem~\eqref{eq:sgbvp} models a strain gradient elasticity~\cite{Altan:1992} with $\boldsymbol{u}$ the displacement field, for which the displacement and the double traction are set free on the boundary of the domain~\cite{Klassen:2011}. Such a model has been considered in many areas such as metamaterials~\cite{Yang:2023} and second-order materials~\cite{Mizel:1998}. It may be viewed as a vector extension of the fourth-order perturbation of the Monge-Amp\`ere equation~\cite{FengNeilan:2014}, the fourth-order singular perturbation problem~\cite{GuzmanLeykekhmanNeilan:2012,Tai:2001,Schuss:1976,Semper:1992} and the models in MEMS~\cite{Laurencot:2017}.

\iffalse Problem~\eqref{eq:sgbvp} is a fourth-order singularly perturbed elliptic system. $C^1$ finite element is a natural choice for discretization though it usually involves a very large number of degrees of freedom and a high polynomial degree of shape functions, and we refer to~\cite{Ak:2006,Fisher:2010,zervos:20092} for numerical tests on $C^1$ element approximation of the strain gradient models. Another choice is the isogeometric analysis~\cite{Hughes:2005} which is popular in engineering applications because it is easy to keep $C^1$  continuity and is flexible in implementation. We refer to~\cite{Klassen:2011,Niiranen:2016} for numerical tests of the strain gradient models and to~\cite{Khakalo:2017} for the higher-order strain gradient models within a commercial finite element software Abaqus. Besides, a series of nonconforming elements have been constructed in~\cite{LiMingShi:2017, LiaoMing:2019, LiMingWang:2021} to approximate problem~\eqref{eq:sgbvp}. These elements converge uniformly in $\iota$ but are not robust in the incompressible limit; i.e., $\lam\to\infty$.\fi

The strain gradient models for the incompressible materials have been studied in~\cite{FleckHutchinson:1997} from the mechanical aspect of view. We refer to~\cite{Bene:2018} for a mathematical study on a nonlinear gradient model for the locking materials. Six mixed finite elements have been presented to approximate the Toupin-Mindlin model for the incompressible materials in~\cite{ShuKingFleck:1999}. Based on a reformulation of the Fleck-Hutchinson model,~\cite{Wei:2006} proposed a novel finite element that successfully predicted the stress field around the crack tip for the incompressible and nearly incompressible materials.~\cite{Fisher:2011} tested a mixed finite element approximation of a simplified version of~\eqref{eq:sgbvp}. Recently, ~\cite{Tian:2021,LiaoMingXu:2023,Huang:2023} proposed several elements to discretize~\eqref{eq:sgbvp} with an essential boundary condition and proved that the proposed elements converge uniformly in both $\iota$ and $\lambda$, i.e., they are {\em locking-free}~\cite{Babuska:1992,Ainsworth:2022}.

In this work, we focus on~\eqref{eq:sgbvp} posed on a convex polygon with the natural boundary condition because it is more practical in engineering~\cite{Wells:2004,Molari:2006,zervos:20091,Fisher:2011}. Following~\cite{LiaoMingXu:2023}, we study a finite element method based on the $\boldsymbol{u}$-$p$ mixed variational formulation with $p=\lam\divop\boldsymbol{u}$, and a saddle point formulation with a penalty term arises. Though the continuous inf-sup condition may not hold for a polygon $\Omega$, we successfully prove a discrete inf-sup condition for certain finite element pair in~Lemma~\ref{lema:infsup}, which ensures the well-posedness of this mixed discretization problem and is key to derive the error bounds.
%This method is nonconforming because the approximation space for %the displacement is $H^2$-nonconforming~\cite{Tai:2001}. 

In the Appendix, using the Agmon-Douglis-Nirenberg (ADN) theory~\cite{ADN:1959, Agmon:1964}, we prove that the continuous inf-sup condition holds true and the $\lambda$-independent estimate 
\[
    \nm{\boldsymbol{u}-\boldsymbol{u}_0}{H^1}+\nm{p-p_0}{L^2}\le C\iota\nm{\boldsymbol{f}}{L^2}
\]
is valid when $\Om$ is smooth, here $\boldsymbol{u}_0$ is the solution of linear elasticity model with homogeneous displacement boundary condition and $p_0=\lambda\divop\boldsymbol{u}_0$. By contrast to the essential boundary value problem studied in~\cite{LiaoMingXu:2023}, we mainly focus on checking that the boundary condition are complementing. Unfortunately, we cannot choose a special outward normal to verify the complementing boundary as in~\cite[Appendix D]{BochevGunzburg:2009} because the system under study is not rotation invariant. The estimate for a polygon $\Omega$ presents significant challenges. Although this issue falls outside the scope of our current discussion, it is related to the regularity estimates in linear elasticity~\cite{BacutaBramble:2003}, which are based on the higher-order regularity results for the Stokes problem. Nonetheless, we may assert that the distance between $\boldsymbol{u}$ and $\boldsymbol{u}_0$ is independent of $\lambda$, which suffices for our paper.

The outline of the paper is as follows. In \S~\ref{sec:Hardy}, we prove the broken Hardy inequality on discontinuous finite element spaces. In~\S~\ref{sec:sge} we construct a nonconforming finite element together with the continuous linear Lagrange finite element to discretize~\eqref{eq:sgbvp}. We prove the nearly optimal rate of convergence that is uniform in both $\lam$ and $\iota$. In \S~\ref{sec:numer}, we report the numerical results, which are consistent with the theoretical predictions. Regularity estimate of~\eqref{eq:sgbvp} is proved for the smooth bounded domain in the Appendix.

Throughout this paper, we assume that the constant $C$ is independent of the mesh size $h$, the materials parameter $\iota$ and the Lam\'e constant $\lambda$. We may drop $\Om$ in $\nm{\cdot}{H^m(\Om)}$  when no confusion occurs. 
\section{A broken Hardy inequality on finite element spaces}\label{sec:Hardy}
Let $B_1^d$ be the unit ball in $\mb R^d$ and $\mc T_h(B_1^d\backslash\{\boldsymbol{0}\})$ be a triangulation of  $B_1^d\backslash\{\boldsymbol{0}\}$ by simplex with maximum mesh size $h$. We assume that all elements $K$ in $\mc T_h$ are shape-regular in the sense of Ciarlet and Raviart~\cite{Ciarlet:2002}, i.e., there exists $\sigma_1>0$ such that
\(h_K\le\sigma_1\rho_K\) for all $K\in\mc{T}_h$, where $h_K=\mathrm{diam}K$ and $\rho_K$ is the diameter of the largest ball inscribed into $K$. In addition, we assume that $\mc T_h$ satisfies the inverse assumption: There exists $\sigma_2>0$ such that
\(h\le\sigma_2h_K\) for all $K\in\mc{T}_h$.

For any $m\in\mathbb{N}$ and $1\le p\le\infty$, the space of piecewise $W^{m,p}$-functions is defined by
\[
W^{m,p}(\mc{T}_h){:}=\set{\boldsymbol f\in L^p(B_1^d)}{\boldsymbol f|_K\in[W^{m,p}(K)]^d\quad\text{for all\quad}K\in\mc{T}_h},
\]
which is equipped with the broken norm
\[
\nm{\boldsymbol f}{W^{m,p}(\mc{T}_h)}^p{:}=\sum_{k=0}^m\nm{\nabla_h^kf}{L^p(B_1^d)}^p\qquad\text{with}\quad
\nm{\na^k_h f}{L^p(B_1^d)}^p{:}=\sum_{K\in\mc{T}_h}\nm{\na^k f}{L^p(K)}^p.
\]
Denote by  $\mc{V}(K)$ the set of all vertices of $K$, and $\mc{V}(\mc{T}_h)$ means the set of all vertices of the triangulation, and let $\mc{F}_h^i$ be the all interior faces (or edges when $d=2$). For an interior face $F\in\mc{F}_h^i$ shared by $K^+$ and $K^-$, and let $\boldsymbol{n}_1$ and $\boldsymbol{n}_2$ be the outward unit normal at $F$ of $K^+$ and $K^-$, respectively. Given $f^+=f|_{K^+}$ and $f^-=f|_{K^-}$, we define the jump of $f$ across $F$ as
\[
\jump{f}|_F{:}=f^+\boldsymbol{n}_1+f^-\boldsymbol{n}_2.
%(f\boldsymbol{n})|_{K^+\cap F}+(f\boldsymbol{n})|_{K^-\cap F}.
\]
For any boundary face (edge) $F$, we set $\jump{f}|_F=f\boldsymbol{n}$.

\iffalse We recall the The following inverse inequality holds~\cite[Theorem 3.2.6, eq. (3.2.33)]{Ciarlet:2002}: There exists a constant $C_{\text{inv}}$ depending on $\sigma_1,\sigma_2,d,r,m,l$ but independent of $h_K$ such that if $m\ge l$, then
\begin{equation}\label{eq:inverse}
    \nm{\nabla^mf}{L^s(K)}\le C_{\text{inv}}h_K^{d(1/s-1/t)-m+l}\nm{\nabla^lf}{L^t(K)}
\end{equation}
for any $f\in\mb P_r(K)$ and $K\in\mc{T}_h$.\fi

%{\red ??This is the key ingredient to control the Hardy inequality on finite element space.}
\begin{theorem}~\label{thm:Hardy}
There exists a constant $C$ depending on $\sigma_1,\sigma_2,d,r$ but independent of $h$ such that for sufficient small $h>0$, there holds 
\begin{equation}\label{eq:thmHardy}
    \nm{f/\rho}{L^d(B_1^d)}\le C\log(1/h)\Lr{\nm{\nabla_hf}{L^d(B_1^d)}+\lr{\sum_{F\in\mc{F}_h^i}\nm{\jump f}{L^\infty(F)}^d}^{1/d}},
\end{equation}
for any $f\in L^d(B_1^d)$ satisfying $f|_K\in\mb{P}_r(K)$ for all $K\in\mc{T}_h(B_1^d\backslash\{\boldsymbol{0}\})$ and $f(\boldsymbol{0})=0$, with a nonnegative integer $r$ and $\rho=\abs{\boldsymbol{x}}$.
\end{theorem}

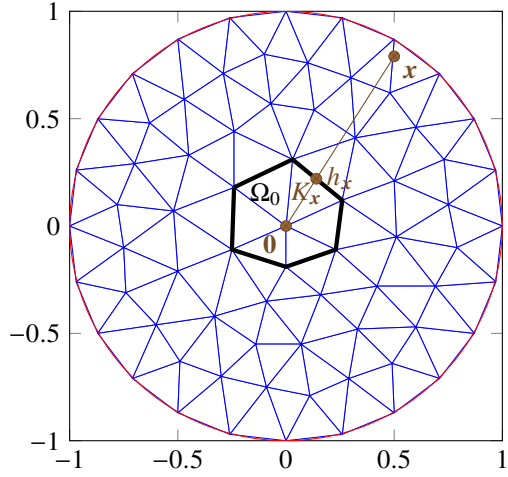
\begin{figure}[htbp]\centering\begin{tikzpicture}\begin{axis}[axis equal image,xmin=-1,xmax=1,ymin=-1,ymax=1]
    \addplot[patch,patch type=triangle,color=white,faceted color=blue]coordinates{(-0.09,-0.35)(-0.25,-0.11)(-0.33,-0.42)(0.42,-0.28)(0.33,-0.48)(0.60,-0.47)(0.33,-0.48)(0.42,-0.28)(0.16,-0.32)(-0.14,-0.55)(-0.33,-0.42)(-0.30,-0.70)(0.42,-0.28)(0.69,-0.28)(0.51,-0.05)(0.23,-0.11)(0.16,-0.32)(0.42,-0.28)(-0.24,0.44)(-0.07,0.57)(-0.29,0.71)(-0.09,-0.35)(-0.14,-0.55)(0.07,-0.55)(0.13,0.54)(0.32,0.41)(0.30,0.69)(0.23,-0.11)(0.42,-0.28)(0.51,-0.05)(-0.24,0.18)(-0.52,0.07)(-0.25,-0.11)(-0.24,0.18)(-0.24,0.44)(-0.45,0.33)(0.03,0.31)(-0.07,0.57)(-0.24,0.44)(0.03,0.31)(-0.24,0.44)(-0.24,0.18)(-0.70,0.29)(-0.87,0.50)(-0.97,0.26)(0.00,0.00)(-0.24,0.18)(-0.25,-0.11)(-0.00,-0.19)(-0.25,-0.11)(-0.09,-0.35)(-0.70,0.29)(-0.63,0.48)(-0.87,0.50)(-0.75,-0.10)(-1.00,-0.00)(-0.97,-0.26)(0.73,0.31)(0.97,0.26)(0.87,0.50)(-0.74,-0.31)(-0.97,-0.26)(-0.87,-0.50)(0.73,0.31)(0.87,0.50)(0.60,0.46)(-0.09,-0.35)(0.07,-0.55)(0.16,-0.32)(0.10,0.76)(0.26,0.97)(-0.00,1.00)(0.30,0.69)(0.50,0.87)(0.26,0.97)(-0.00,-0.19)(0.00,0.00)(-0.25,-0.11)(0.48,0.62)(0.71,0.71)(0.50,0.87)(0.23,-0.11)(0.51,-0.05)(0.26,0.12)(-0.00,-0.19)(-0.09,-0.35)(0.16,-0.32)(0.30,0.69)(0.26,0.97)(0.10,0.76)(-0.78,0.10)(-0.97,0.26)(-1.00,-0.00)(-0.30,-0.70)(-0.50,-0.87)(-0.26,-0.97)(0.03,0.31)(0.32,0.41)(0.13,0.54)(0.03,0.31)(0.26,0.12)(0.32,0.41)(0.03,0.31)(0.13,0.54)(-0.07,0.57)(0.00,0.00)(0.23,-0.11)(0.26,0.12)(0.00,0.00)(0.26,0.12)(0.03,0.31)(-0.74,-0.31)(-0.87,-0.50)(-0.60,-0.46)(0.00,0.00)(0.03,0.31)(-0.24,0.18)(-0.00,-0.19)(0.16,-0.32)(0.23,-0.11)(-0.00,-0.19)(0.23,-0.11)(0.00,0.00)(0.48,0.62)(0.60,0.46)(0.71,0.71)(0.73,0.31)(0.75,0.10)(0.97,0.26)(0.79,-0.10)(0.97,-0.26)(1.00,0.00)(0.69,-0.28)(0.87,-0.50)(0.97,-0.26)(0.60,-0.47)(0.71,-0.71)(0.87,-0.50)(-0.09,-0.75)(-0.26,-0.97)(0.00,-1.00)(-0.09,-0.75)(0.00,-1.00)(0.11,-0.79)(0.11,-0.79)(0.00,-1.00)(0.26,-0.97)(0.60,-0.47)(0.49,-0.64)(0.71,-0.71)(-0.63,0.48)(-0.71,0.71)(-0.87,0.50)(-0.29,0.71)(-0.26,0.97)(-0.50,0.87)(0.49,-0.64)(0.50,-0.87)(0.71,-0.71)(-0.10,0.80)(-0.26,0.97)(-0.29,0.71)(-0.63,0.48)(-0.47,0.59)(-0.71,0.71)(-0.49,-0.63)(-0.71,-0.71)(-0.50,-0.87)(-0.74,-0.31)(-0.75,-0.10)(-0.97,-0.26)(-0.07,0.57)(0.10,0.76)(-0.10,0.80)(0.29,-0.71)(0.26,-0.97)(0.50,-0.87)(0.29,-0.71)(0.11,-0.79)(0.26,-0.97)(-0.49,-0.63)(-0.60,-0.46)(-0.71,-0.71)(-0.60,-0.46)(-0.87,-0.50)(-0.71,-0.71)(-0.09,-0.75)(-0.30,-0.70)(-0.26,-0.97)(-0.49,-0.63)(-0.50,-0.87)(-0.30,-0.70)(-0.47,0.59)(-0.50,0.87)(-0.71,0.71)(-0.47,0.59)(-0.29,0.71)(-0.50,0.87)(-0.07,0.57)(-0.10,0.80)(-0.29,0.71)(0.60,0.46)(0.87,0.50)(0.71,0.71)(-0.50,-0.21)(-0.74,-0.31)(-0.60,-0.46)(0.60,-0.47)(0.87,-0.50)(0.69,-0.28)(-0.33,-0.42)(-0.60,-0.46)(-0.49,-0.63)(0.50,0.21)(0.75,0.10)(0.73,0.31)(0.49,-0.64)(0.29,-0.71)(0.50,-0.87)(-0.24,0.44)(-0.29,0.71)(-0.47,0.59)(0.79,-0.10)(0.69,-0.28)(0.97,-0.26)(-0.24,0.44)(-0.47,0.59)(-0.45,0.33)(0.75,0.10)(1.00,0.00)(0.97,0.26)(0.75,0.10)(0.79,-0.10)(1.00,0.00)(0.42,-0.28)(0.60,-0.47)(0.69,-0.28)(0.51,-0.05)(0.69,-0.28)(0.79,-0.10)(0.51,-0.05)(0.79,-0.10)(0.75,0.10)(-0.78,0.10)(-0.70,0.29)(-0.97,0.26)(-0.45,0.33)(-0.47,0.59)(-0.63,0.48)(0.07,-0.55)(0.11,-0.79)(0.29,-0.71)(-0.10,0.80)(-0.00,1.00)(-0.26,0.97)(-0.10,0.80)(0.10,0.76)(-0.00,1.00)(0.07,-0.55)(-0.09,-0.75)(0.11,-0.79)(0.26,0.12)(0.50,0.21)(0.32,0.41)(-0.14,-0.55)(-0.30,-0.70)(-0.09,-0.75)(-0.33,-0.42)(-0.50,-0.21)(-0.60,-0.46)(-0.50,-0.21)(-0.75,-0.10)(-0.74,-0.31)(-0.45,0.33)(-0.63,0.48)(-0.70,0.29)(0.48,0.62)(0.50,0.87)(0.30,0.69)(-0.33,-0.42)(-0.49,-0.63)(-0.30,-0.70)(-0.78,0.10)(-1.00,-0.00)(-0.75,-0.10)(0.51,-0.05)(0.75,0.10)(0.50,0.21)(0.33,-0.48)(0.49,-0.64)(0.60,-0.47)(0.32,0.41)(0.48,0.62)(0.30,0.69)(-0.24,0.18)(-0.45,0.33)(-0.52,0.07)(0.32,0.41)(0.60,0.46)(0.48,0.62)(-0.25,-0.11)(-0.52,0.07)(-0.50,-0.21)(-0.52,0.07)(-0.78,0.10)(-0.75,-0.10)(0.50,0.21)(0.73,0.31)(0.60,0.46)(0.50,0.21)(0.60,0.46)(0.32,0.41)(0.13,0.54)(0.30,0.69)(0.10,0.76)(0.13,0.54)(0.10,0.76)(-0.07,0.57)(-0.52,0.07)(-0.75,-0.10)(-0.50,-0.21)(-0.52,0.07)(-0.70,0.29)(-0.78,0.10)(0.33,-0.48)(0.29,-0.71)(0.49,-0.64)(-0.14,-0.55)(-0.09,-0.75)(0.07,-0.55)(0.33,-0.48)(0.07,-0.55)(0.29,-0.71)(-0.25,-0.11)(-0.50,-0.21)(-0.33,-0.42)(-0.52,0.07)(-0.45,0.33)(-0.70,0.29)(0.16,-0.32)(0.07,-0.55)(0.33,-0.48)(0.26,0.12)(0.51,-0.05)(0.50,0.21)(-0.09,-0.35)(-0.33,-0.42)(-0.14,-0.55)};
    \addplot[red,samples=100, domain=0:2*pi]( {cos(deg(x))},{sin(deg(x))});
    \addplot coordinates{(0,0)(0.14,0.22)(0.5,0.79)}node[pos=0,anchor=north east](A){$\boldsymbol{0}$}node[pos=.19]{$K_{\boldsymbol{x}}$}node[pos=.28,anchor=west]{$h_{\boldsymbol{x}}$}node[pos=1,anchor=north west]{$\boldsymbol{x}$};
    \addplot[ultra thick,mark=none]coordinates{(-0.24,0.18)(0.03,0.31)(0.26,0.12)(0.23,-0.11)(-0.00,-0.19)(-0.25,-0.11)(-0.24,0.18)}node[pos=.1,anchor=north]{$\Omega_0$};
\end{axis}\end{tikzpicture}\caption{Diagram for the proof of the broken Hardy inequality when $d=2$}\label{fig:Hardy}\end{figure}

\begin{proof}
Firstly, we assume that $f$ is continuous. For any $\boldsymbol{x}\in B_1^d$,
\begin{equation}\label{eq:thmHardy0}
     f(\boldsymbol{x})=\boldsymbol{x}\cdot\int_0^1\nabla f(t\boldsymbol{x})\mathrm{d}t.
\end{equation}
%Let $\mc{T}_{\boldsymbol{a}}{:}=\set{K\in\mc T_h}{\boldsymbol{a}\in K}$ 
Let $\mc{T}_{\boldsymbol{a}}$ be the set of elements sharing a common vertex $\boldsymbol{a}$ and $\Omega_0{:}=\cup_{K\in\mc T_{\boldsymbol{0}}}K$. First, for any $K\in\mc T_{\boldsymbol{0}}$ and $\boldsymbol{x}\in K$, using the inverse inequality, we obtain
\[
\abs{f(\boldsymbol{x})}\le\rho\nm{\nabla f}{L^\infty(K)}\le C_{\text{inv}}h_K^{-1}\rho\nm{\nabla f}{L^d(K)}.
\]
%by the inverse equality~\eqref{eq:inverse} and~\eqref{eq:thmHardy0}. Then
This immediately implies
\[
\nm{f/\rho}{L^d(K)}\le C_{\text{inv}}\abs{K}^{1/d}h_K^{-1}\nm{\nabla f}{L^d(K)}\le C_{\text{inv}}\nm{\nabla f}{L^d(K)}.
\]
Summing up all $K\in\mc T_{\boldsymbol{0}}$, we obtain
\begin{equation}\label{eq:thmHardy1}
    \nm{f/\rho}{L^d(\Omega_0)}%=\Lr{\sum_{K\in\mc T_{\boldsymbol{0}}}\nm{f/\rho}{L^d(K)}^d}^{1/d}
    \le C_{\text{inv}}\nm{\nabla_hf}{L^d(\Omega_0)}.
\end{equation}

Second, for any $\boldsymbol{x}\not\in\Omega_0$, let $K_{\boldsymbol{x}}\in\mc T_{\boldsymbol{0}}$ be the first element that the line segment  $\overline{\boldsymbol{0x}}$ passes through, and $h_{\boldsymbol{x}}$ be the length between $\boldsymbol{0}$ and the intersection point of $\overline{\boldsymbol{0x}}$ and $\partial K_{\boldsymbol{x}}$; see Figure~\ref{fig:Hardy} for a diagram when $d=2$. We may write~\eqref{eq:thmHardy0} as
\[
f(\boldsymbol{x})=\boldsymbol{x}\cdot\int_0^{h_{\boldsymbol{x}}/\rho}\nabla f(t\boldsymbol{x})\mathrm{d}t+\boldsymbol{x}\cdot\int_{h_{\boldsymbol{x}}/\rho}^1\nabla f(t\boldsymbol{x})\mathrm{d}t.
\]
%Using the inverse inequality again since $t\boldsymbol{x}\in K_{\boldsymbol{x}}$ for any $t\in (0,h_{\boldsymbol{x}}/\rho)$, we obtain
Proceeding along the same line that leads to~\eqref{eq:thmHardy1} because $t\boldsymbol{x}\in K_{\boldsymbol{x}}$ for any $t\in (0,h_{\boldsymbol{x}}/\rho)$, we obtain
\[
\abs{f(\boldsymbol{x})}\le C_{\text{inv}}\nm{\nabla f}{L^d(K_{\boldsymbol{x}})}+\rho\int_{ch}^1\abs{\nabla f(t\boldsymbol{x})}\mathrm{d}t
\]
with $c=2/(\sigma_1\sigma_2)$. By the generalized Minkowski inequality,
\begin{align*}
    \nm{f/\rho}{L^d(B_1^d\backslash\Omega_0)}&\le C_{\text{inv}}\nm{\nabla f}{L^d(\Omega_0)}\nm{1/\rho}{L^d(B_1^d\backslash\Omega_0)}+\int_{ch}^1\nm{\nabla f(t\cdot)}{L^d(B_1^d\backslash\Omega_0)}\mathrm{d}t\\
    &\le C_{\text{inv}}\omega_d^{1/d}\log^{1/d}(1/ch)\nm{\nabla f}{L^d(\Omega_0)}+\int_{ch}^1\dfrac{1}{t}\mathrm{d}t\nm{\nabla f}{L^d(B_1^d)}\\
    &\le\Lr{C_{\text{inv}}\omega_d^{1/d}\log^{1/d}(1/ch)+\log(1/ch)}
    \nm{\nabla f}{L^d(B_1^d)}.
\end{align*}
Combining the above inequality with~\eqref{eq:thmHardy1}, we obtain
\[
\nm{f/\rho}{L^d(B_1^d)}\le 2\log(1/ch)\nm{\nabla f}{L^d(B_1^d)}
\]
for sufficiently small $h$, e.g., $h<h_0$ with
\[
h_0=\dfrac{\sigma_1\sigma_2}{2}\min\Lr{e^{-1/\omega_d},e^{-(2C_{\text{inv}}\omega_d^{1/d})^{d'}}}.
\]

In what follows, we extend~\eqref{eq:thmHardy} to the  continuous piecewise function $f$. We only need to prove the inequality over $B_1^d\backslash\Omega_0$ because~\eqref{eq:thmHardy1} remains true. For each $f$, we define a continuous relative $f_I$ as: 
%an enriching operator from $f$ to the  function in the Lagrange finite element space
\[
f_I\in\set{f\in C(B_1^d)}{f|_K\in\mb{P}_1(K)\quad\text{for all }K\in\mc{T}_h(B_1^d\backslash\{\boldsymbol{0}\})},
\]
uniquely determined by
\[
    f_I(\boldsymbol{a})=\dfrac{1}{\abs{\mc{T}_{\boldsymbol{a}}}}\sum_{K\in\mc{T}_{\boldsymbol{a}}}f|_K(\boldsymbol{a})\qquad\text{for all }\boldsymbol{a}\in\mc{V}(\mc{T}_h(B_1^d\backslash\{\boldsymbol{0}\})).
\]
$f_I$ is continuous and $f_I(\boldsymbol{0})=f(\boldsymbol{0})=0$, and for sufficiently small $h$, 
\[
\nm{f_I/\rho}{L^d(B_1^d)}\le 2\log(1/ch)\nm{\nabla f_I}{L^d(B_1^d)}.
\]

It remains to evaluate the distance between $f$ and $f_I$, which may be proceeded as in~\cite{Brenner:2004Sobolev}. %,Brenner:2004Korn,Brenner:2004Poincare}. 
For any $K\in\mc{T}_h(B_1^d\backslash\{\boldsymbol{0}\})$ and $\boldsymbol{a}\in\mc{V}(K)$, there exists $C$ depends only on $\sigma_1$ such that
\[
(f|_K-f_I)(\boldsymbol{a})=\dfrac{1}{\abs{\mc{T}_{\boldsymbol{a}}}}\sum_{K'\in\mc{T}_{\boldsymbol{a}}}(f|_K-f|_{K'})(\boldsymbol{a})\le\sum_{F\in\mc{F}_h^i\text{ such that }\boldsymbol{a}\in F}\nm{\jump f}{L^\infty(F)}.
\]

On each element $K$, we define the local interpolant 
$\mathcal{I}_Kf$ by $\mathcal{I}_Kf(\boldsymbol{a})=f(\boldsymbol{a})$ for $\boldsymbol{a}\in\mc{V}(K)$. Using the scaling trick and the standard interpolate estimate, there exists $C$ depending on $\sigma_1$ and $d$ such that
\begin{align*}
\nm{f-f_I}{L^d(K)}&\le\nm{f-\mathcal{I}_Kf}{L^d(K)}+\nm{\mathcal{I}_Kf-f_I}{L^d(K)}\\
&\le Ch\Lr{\nm{\nabla f}{L^d(K)}+\sum_{F\in\mc{F}_h^i\text{ such that }\boldsymbol{a}\in F}\nm{\jump f}{L^\infty(F)}}.
\end{align*}
Summing up for all $K\in\mc{T}_h(B_1^d\backslash\{\boldsymbol{0}\})$, we obtain
\begin{equation}\label{eq:thmHardy2}
\nm{f-f_I}{L^d(B_1^d)}\le Ch\Lr{\nm{\nabla_hf}{L^d(B_1^d)}+\lr{\sum_{F\in\mc{F}_h^i}\nm{\jump f}{L^\infty(F)}^d}^{1/d}},
\end{equation}
which together with the triangle inequality and the inverse inequality, we obtain, for sufficiently small $h$, there holds
\begin{align*}
\nm{f/\rho}{L^d(B_1^d\backslash\Omega_0)}&\le\nm{f_I/\rho}{L^d(B_1^d\backslash\Omega_0)}+\nm{(f-f_I)/\rho}{L^d(B_1^d\backslash\Omega_0)}\\
&\le 2\log(1/ch)\nm{\nabla f_I}{L^d(B_1^d)}+ch^{-1}\nm{f-f_I}{L^d(B_1^d\backslash\Omega_0)}\\
&\le 2\log(1/ch)\nm{\nabla_hf}{L^d(B_1^d)}+(2C_{\text{inv}}\log(1/ch)+c)h^{-1}\nm{f-f_I}{L^d(B_1^d)}.
\end{align*}
Substituting \eqref{eq:thmHardy2} into the above inequality, and using~\eqref{eq:thmHardy1}, we obtain~\eqref{eq:thmHardy}.
\end{proof}

\begin{remark}
We conjecture that the logarithmic factor in the above theorem is optimal. A typical example is a continuous radial function $f(\boldsymbol{x})=f_0(\rho)$ satisfying $f_0(0)=0$ and $f_0^{\prime}=1/i$ for $(i-1)h<\rho<ih$, here $i=1,2,\dots,n$ and $n=1/h$. Then, for $(i-1)h\le\rho<ih$, 
\[
f_0(\rho)=\dfrac{\rho}{i}+h\sum_{j=1}^{i-1}\dfrac{1}{j}\ge h\log i,\qquad i=1,2,\dots n.
\]
A direct calculation gives 
\begin{align*}
\dfrac{1}{\omega_d}\int_{B_1^d}\dfrac{\abs{f}^d}{\abs{\boldsymbol{x}}^d}\mathrm{d}\boldsymbol{x}
&=\int_0^1\dfrac{\abs{f_0}^d}{\rho}\mathrm{d}\rho\ge h^d\sum_{i=2}^n\log^di\int_{(i-1)h}^{ih}\dfrac{1}{\rho}\mathrm{d}\rho\\
&\ge h^d\sum_{i=2}^n\dfrac{1}{i}\log^di\\
&\ge\dfrac{h^d}{d+1}(\log^{d+1}n-1),
\end{align*}
and
\begin{align*}
\dfrac{1}{\omega_d}\int_{B_1^d}\abs{\nabla f}^d\mathrm{d}\boldsymbol{x}&=\int_0^1\rho^{d-1}\abs{f_0'}^d\mathrm{d}\rho=\sum_{i=1}^n\int_{(i-1)h}^{ih}\dfrac{\rho^{d-1}}{i^d}\mathrm{d}\rho\\
&\le h^d\sum_{i=1}^n\dfrac{1}{i}\le h^d(\log n+1).
\end{align*}
Therefore, $\nm{f/\rho}{L^d(B_1^d)}\sim\mc{O}(h\log^{1+1/d}(1/h))$ and $\nm{\nabla f}{L^d(B_1^d)}\sim\mc{O}(h\log^{1/d}(1/h))$. 

When $d=1$, this conjecture is clear since $f|_K\in\mb{P}_1(K)$ for all $K\in\mc{T}_h((-1,0)\cup(0,1))$ strictly. When $d\ge 2$, this example remains instructive since $f$, when restricted to any radial direction, is still a finite element function.
\end{remark}
\begin{coro}
%Let $W_h\subset L^{d}(B_1^d)$ is a finite dimensional space, 
Let $K\in\mc{T}_h(B_1^d\backslash\{\boldsymbol{0}\}), f|_{K}\in\mb P_r(K)$ and $f(\boldsymbol{0})=0$. Assume further that $f$ satisfies the \textbf{consistency condition}:   
\begin{equation}\label{def:csst}
\nm{\jump{f}}{L^\infty(F)}\le Ch\nm{\nabla_hf}{L^\infty(K^+\cup K^-)}\quad\text{for all }F\in\mc F_h^i,
\end{equation}
where $K^+,K^-\in\mc{T}_h(B_1^d\backslash\{\boldsymbol{0}\})$ and $F=K^+\cap K^-$. There exists a constant $C$ depending only on $\sigma_1,\sigma_2,d,r$ but independent of $h$ such that 
\begin{equation}\label{eq:hardy3}
    \nm{f/\rho}{L^d(B_1^d)}\le C\log(1/h)\nm{\nabla_hf}{L^d(B_1^d)}.
\end{equation}
%where $\rho=\abs{\boldsymbol{x}}$.
\end{coro}

\begin{proof}
By the consistency condition~\eqref{def:csst} and the inverse inequality for $f$,
\[
\lr{\sum_{F\in\mc F_h^i}\nm{\jump{f}}{L^\infty(F)}^d}^{1/d}\le C\lr{\sum_{F\in\mc F_h^i}\nm{\nabla_hf}{L^d(K^+\cup K^-)}^d}^{1/d}\le C\nm{\nabla_hf}{L^d(B_1^d)}.
\]
Therefore, the second term in~\eqref{eq:thmHardy} is bounded by the first one, and we get~\eqref{eq:hardy3}.
%complete the proof.
\end{proof}

\begin{remark}
Most of the commonly used elements satisfy the consistency condition~\eqref{def:csst}, e.g. the Lagrange element, the Crouzeix-Raviart element and the Morley triangle.
\end{remark}
\section{Application for strain gradient elasticity}~\label{sec:sge}
For a bounded polygon $\Omega\subset\mb R^2$, the model~\eqref{eq:sgbvp} is equivalent to the variational problem: Find $u\in V{:}=[H_0^1(\Om)\cap H^2(\Om)]^2$ such that
\begin{equation}\label{eq:var}
    a(\boldsymbol{u},\boldsymbol{v}){:}=(\C\boldsymbol\eps(\boldsymbol u),\boldsymbol\eps(\boldsymbol v))+\iota^2(\D\na\boldsymbol\eps(\boldsymbol u),\na\boldsymbol\eps(\boldsymbol v))=(\boldsymbol{f},\boldsymbol{v})\qquad\text{for all}\quad\boldsymbol{v}\in V,
\end{equation}
where the sixth-order tensor $\D$ is defined as
\[
\D_{ijklmn}=\lam\del_{il}\del_{jk}\del_{mn}+2\mu\del_{il}\del_{jm}\del_{kn}.
\]
The strain gradient $\na\boldsymbol\eps(\boldsymbol v)$ is a third-order tensor given by $(\na\boldsymbol\eps(\boldsymbol v))_{ijk}=\partial_i\boldsymbol\eps(\boldsymbol v)_{jk}$.

Let $\iota\to 0$, we find $\boldsymbol{u}_0\in[H_0^1(\Om)]^2$ satisfying
\begin{equation}\label{eq:elas}
(\mb{C}\boldsymbol{\epsilon}(\boldsymbol{u}_0),\boldsymbol{\epsilon}(\boldsymbol{v}))=(\boldsymbol{f},\boldsymbol{v})\qquad\text{for all\;}\boldsymbol{v}\in [H_0^1(\Om)]^2.
\end{equation}
Let $p_0=\lambda\divop\boldsymbol{u}_0$. By~\cite[Theorem 3.1]{BacutaBramble:2003},
\begin{equation}\label{eq:regelas}
\nm{\boldsymbol{u}_0}{H^2}+{\nm{p_0}{H^1}}\le C\nm{\boldsymbol{f}}{L^2}.
\end{equation}

\iffalse Throughout this paper, we shall frequently use the following Poincar\'e inequality: There exists $C_P$ only depending on $\Om$ such that
\begin{align*}
\nm{v}{H^2}&\le C_P\nm{\na^2 v}{L^2} && v\in H_0^1(\Om)\cap H^2(\Om),\\
\nm{v}{H^1}&\le C_P\nm{\nabla v}{L^2} && v\in L_0^2(\Om)\cap H^1(\Om).
\end{align*}\fi

For any $q\in L^2(\Omega)\cap H^1(\Omega)$, we define a weighted norm
\(
\nm{q}{\iota}{:}=\nm{q}{L^2}+\iota\nm{\na q}{L^2}.
\)
By the Poincar\'e inequality, $\nm{\nabla\boldsymbol{v}}{\iota}$ is indeed a norm over $V$ for any $\boldsymbol{v}\in V$.

\begin{lemma}\label{lema:korn}
There holds
\begin{equation}\label{eq:1stkorn}
\nm{\boldsymbol{\epsilon}(\boldsymbol{v})}{L^2}^2\ge\dfrac12\nm{\nabla\boldsymbol{v}}{L^2}^2\qquad\text{for all\quad}\boldsymbol{v}\in [H_0^1(\Om)]^2,
\end{equation}
and 
\begin{equation}\label{eq:h2korn}
\nm{\na\boldsymbol{\epsilon}(\boldsymbol{v})}{L^2}^2\ge\Lr{1-1/\sqrt2}\nm{\na^2\boldsymbol{v}}{L^2}^2\qquad\text{for all\quad}\boldsymbol{v}\in [H^2(\Om)]^2.
\end{equation}
\end{lemma}

The first Korn's inequality~\eqref{eq:1stkorn} may be found in~\cite{Korn:1908}, and the H$^2$-Korn's inequality~\eqref{eq:h2korn} was proved in~\cite[Theorem 1]{LiMingWang:2021}. By Lax-Milgram theorem and Lemma~\ref{lema:korn}, Problem~\eqref{eq:var} admits a unique solution $\boldsymbol{u}$.
\iffalse and
\[
\nm{\nabla(\boldsymbol{u}-\boldsymbol{u}_0)}{\iota}\le C\iota\nm{\boldsymbol{f}}{L^2}.??
\]\fi

\subsection{The mixed variational formulation}
Let $P{:}=\divop V$. This space has been clarified in~\cite{Arnold:1988} as $P=L_0^2(\Om)\cap H^1_+(\Om)$ with the semi-norm
\[
[q]_{H_+^1}^2=\sum_{\boldsymbol{a}\in\mc{V}(\Omega)}\int_\Omega\dfrac{\abs{q(\boldsymbol{x})}^2}{\abs{\boldsymbol{x}-\boldsymbol{a}}^2}\mathrm{d}\boldsymbol{x},
\]
where $\mc{V}(\Omega)$ consists of all vertices of $\Omega$, and
\(
\nm{q}{H_+^1}=\nm{q}{H^1}+[q]_{H_+^1}.
\)

\begin{lemma}[{~\cite[Theorem 3.1]{Arnold:1988}}]\label{lema:divP}
If $\Omega\subset\mb{R}^2$ is a polygon, then for any $q\in P$, there exists $\boldsymbol{v}\in V$ such that $\divop\boldsymbol{v}=q$ and
\[\nm{\boldsymbol{v}}{H^1}\le C\nm{q}{L^2},\quad\nm{\boldsymbol{v}}{H^2}\le C\nm{q}{H_+^1}.\]
\end{lemma}

We introduce $p{:}=\lam\divop\boldsymbol{u}$ and write~\eqref{eq:sgbvp} into a mixed variational problem as
\begin{equation}\label{eq:mix}
\left\{
\begin{aligned}
a_{\iota}(\boldsymbol{u},\boldsymbol{v})+b_\iota(\boldsymbol{v},p)&=(\boldsymbol{f},\boldsymbol{v})\qquad&&\text{for all\quad}\boldsymbol{v}\in V,\\
b_{\iota}(\boldsymbol{u},q)-\lam^{-1}c_{\iota}(p,q)&=0\qquad&&\text{for all\quad}q\in P,
\end{aligned}\right.
\end{equation}
with
\begin{align*}
a_{\iota}(\boldsymbol{u},\boldsymbol{v}){:}&=2\mu\lr{(\boldsymbol{\epsilon}(\boldsymbol{u}),\boldsymbol{\epsilon}(\boldsymbol{v}))+\iota^2(\na\boldsymbol{\epsilon}(\boldsymbol{u}),\na\boldsymbol{\epsilon}(\boldsymbol{v}))},\\
 b_\iota(\boldsymbol{v},p){:}&=(\divop\boldsymbol{v},p)+\iota^2(\na\divop\boldsymbol{v},\na p),\quad\text{and\quad} c_{\iota}(p,q){:}=(p,q)+\iota^2(\na p,\na q).
\end{align*}

\begin{lemma}
The variational problem~\eqref{eq:var} and the mixed variational problem~\eqref{eq:mix} are equivalent for any fixed $\lambda$.
\end{lemma}
This is obvious by choosing $q=\divop\boldsymbol{u}-\lambda^{-1}p\in P$. Therefore~\eqref{eq:mix} admits a unique solution $(\boldsymbol{u},p)=(\boldsymbol{u},\lambda\divop\boldsymbol{u})$.

If $\Omega$ is smooth, %then for the solution $(\boldsymbol{u},p)$ of Problem~\eqref{eq:mix}, there exists $C$ that is independent of $\iota$ and $\lambda$ such that 
%\[
%\nm{\nabla\boldsymbol{u}}{\iota}+\nm{p}{\iota}\le C\nm{\boldsymbol{f}}{H^{-1}}.
%\]
%We shall prove the above inequality in~\eqref{eq:estsolu}.
then we shall prove 
\[
\nm{\boldsymbol u-\boldsymbol{u}_0}{H^1}+\nm{p-p_0}{L^2}\le C\iota\nm{\boldsymbol f}{L^2},
\]
and
\[
\nm{\boldsymbol u}{H^3}+\nm{p}{H^2}\le C\iota^{-1}\nm{\boldsymbol f}{L^2}
\]
in Theorem~\ref{thm:reg1} and~\ref{thm:reg2}. Moreover, if $\boldsymbol{u}_0$ and $p_0$ are smoother, then we shall prove
\[
\nm{\nabla(\boldsymbol u-\boldsymbol{u}_0)}{\iota}+\nm{p-p_0}{\iota}\le C\iota^{3/2}(\nm{\boldsymbol{u}_0}{H^3}+\nm{p_0}{H^2}),
\]
and
\[
\nm{\boldsymbol u}{H^3}+\nm{p}{H^2}\le C\iota^{-1/2}(\nm{\boldsymbol{u}_0}{H^3}+\nm{p_0}{H^2})
\]
in the Appendix~\ref{appd:sge}.
\subsection{A nonconforming finite element}
We introduce a finite element to approximate Problem~\eqref{eq:mix}. Let $\T_h(\Omega)$ be a shape-regular, simplical triangulation of $\Om$ in the sense of Ciarlet and Raviart~\cite{Ciarlet:2002} and satisfying the inverse assumption. We denote by $\mc{E}_h$ the set of all edges in $\mc{T}_h$, $\mc{E}_h^i$ the set of edges in the interior of $\Omega$, and $\mc{E}_h^b$ the set of edges on $\pa\Om$. 

Define a finite element with a triple $(K,P_K,\Sigma_K)$ by specifying $K$ as a triangle and 
\[
P_K{:}=\mb{P}_2(K)+b_K\text{span}\{\mb{P}_1(K)+\lam_1\lam_2+\lam_2\lam_3+\lam_3\lam_1\},
\]
where $\{\lam_i\}_{i=1}^3$ are the barycentrical coordinates of $K$, and $b_K=\lam_1\lam_2\lam_3$ is the bubble function. Define the set of the degrees of freedom (DoFs) as
\begin{equation}\label{eq:dom}
\Sigma_K{:}=\set{v(\boldsymbol{a}_i),v(\boldsymbol{b}_i),\int_{e_i}\partial_{\boldsymbol{n}} v\md\sigma(\boldsymbol{x}),\int_K v\dx}{i=1,2,3},
\end{equation}
where $\{\boldsymbol{a}_i\}_{i=1}^3$ are three vertices of $K$, and $\{\boldsymbol{b}_i\}_{i=1}^3$ are three mid-points associated with the edges $e_i$. This element may be viewed as a modified version of~\cite{Tai:2001} (cf. Figure~\ref{fig:Diagram}).
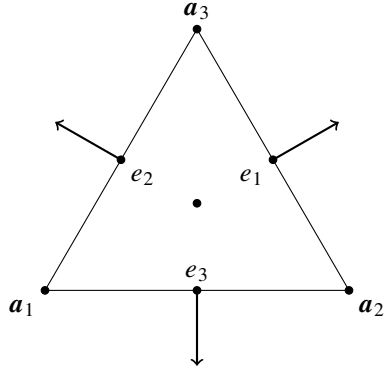
\begin{figure}[htbp]\centering\begin{tikzpicture}
\draw(-2,0)--(2,0)--(0,3.464)--(-2,0);
\draw[fill](-2,0)circle[radius=0.05];
\draw[fill](2,0)circle[radius=0.05];
\draw[fill](0,3.464)circle[radius=0.05];
\draw[fill](1,1.732)circle[radius=0.05];
\draw[fill](-1,1.732)circle[radius=0.05];
\draw[fill](0,0)circle[radius=0.05];
\draw[->][thick](1,1.732)--(1.866,2.232);
\draw[->][thick](-1,1.732)--(-1.866,2.232);
\draw[->][thick](0,0)--(0,-1);
\draw[fill](0,1.155)circle[radius=0.05];
\node[below left]at(-2,0){$\boldsymbol{a}_1$};
\node[below right]at(2,0){$\boldsymbol{a}_2$};
\node[above]at(0,3.47){$\boldsymbol{a}_3$};
\node[below left]at(1,1.732){$e_1$};
\node[below right]at(-1,1.732){$e_2$};
\node[above]at(0,0){$e_3$};
\end{tikzpicture}\caption{Diagram for the degrees of freedom.}\label{fig:Diagram}
\end{figure}
\begin{lemma}
The set $(K,P_K,\Sigma_K)$ is unisolvent.
\end{lemma}

\begin{proof}
The dimension of $P_K$ is 10, which equals the number of DoFs~\eqref{eq:dom} exactly. Therefore, for any $v\in P_K$, we only need to show that if all DoFs vanish, then $v\equiv 0$. Define
\[
\phi_0{:}=140b_K\Lr{5(\lambda_1\lambda_2+\lambda_2\lambda_3+\lambda_3\lambda_1)-1}.
\]
It is obvious that $\phi_0$ vanishes on all DoFs except the last one, and $\int_K\phi_0(x)\dx=\abs{K}$. We express $v$ as
\(v=p_2+b_Kp_1+c\phi_0,\)
where $p_2\in\mb{P}_2(K)$, $p_1\in\mb{P}_1(K)$ and $c\in\mb{R}$. Since $v|_{e_i}=p_2|_{e_i}\in\mb{P}_2(e_i)$ with three roots, we must have $p_2|_{e_i}=0$. Therefore $p_2\equiv 0$. 

A direct calculation gives
\[
\int_{e_i}\partial_{\boldsymbol{n}}v\dsx=-\abs{\lambda_i}\int_{e_i}p_1\lambda_j\lambda_k\dsx=0.
\]
Therefore $p_1$ has a root in the interior of $e_i$. Hence $p_1$ has three roots over $\pa K$, and $p_1\equiv 0$. Finally $v$ vanishes on the last degrees of freedom shows that $c=0$. Thus $v\equiv 0$.
\end{proof}

We list the shape functions of the element. The one associated with $\negint_K v\dx$ is
\[
\phi_0=140b_K\Lr{5(\lambda_1\lambda_2+\lambda_2\lambda_3+\lambda_3\lambda_1)-1}.
\]
The shape functions associated with $\{\negint_{e_i}\partial_{\boldsymbol{n}}v\md\sigma(\boldsymbol{x})\}_{i=1,2,3}$ are
\[
\psi_i=\dfrac{6}{\abs{\nabla\lambda_i}}b_K(2\lambda_i-1)+\dfrac1{30\abs{\nabla\lambda_i}}\phi_0.
\]
The shape functions associated with $\{v(\boldsymbol{a}_i)\}_{i=1,2,3}$ are
\[
\phi_i=\lambda_i(2\lambda_i-1)-\abs{\na\lam_i}\psi_i+\sum_{j\neq i}\dfrac{\nabla\lambda_i\cdot\nabla\lambda_j}{\abs{\nabla\lambda_j}}\psi_j.
\]
The shape functions associated with $\{v(\boldsymbol{b}_i)\}_{i=1,2,3}$ are
\[
\varphi_i=4\lambda_j\lambda_k+12b_K(1-4\lambda_i)-\dfrac4{15}\phi_0.
\]

We define a local interpolation operator $\pi_K:H^2(K)\to P_K$ as:
\begin{equation}\label{eq:inter1}
\left\{\begin{aligned}
\pi_K v(\boldsymbol{a})&=v(\boldsymbol{a})\quad&&\text{for all vertices }\boldsymbol{a},\\
\pi_K v(\boldsymbol{b})&=v(\boldsymbol{b})\quad&&\text{for all mid-points }\boldsymbol{b}\text{ of each edges},\\
\int_e\partial_{\boldsymbol{n}}\pi_K v\md\sigma(\boldsymbol{x})&=\int_e\partial_{\boldsymbol{n}}v\md\sigma(\boldsymbol{x})\quad&&\text{for all edges }e,\\
\int_K\pi_Kv\dx&=\int_Kv\dx.
\end{aligned}\right.
\end{equation}
\begin{lemma}
For $v\in H^s(K)$ with $2\le s\le 3$, there exists $C$ independent of $h_K$ such that 
\begin{equation}\label{eq:localerr}
\nm{\na^j(v-\pi_Kv)}{L^2(K)}\le Ch_K^{s-j}\nm{v}{H^s(K)}\qquad j=0,1,2.
\end{equation}
\end{lemma}

\begin{proof}
For any $v\in\mb{P}_2(K)\subset P_K$, the definition~\eqref{eq:inter1} shows that $v-\pi_Kv\in P_K$ and the degrees of freedom vanish, hence $v=\pi_Kv$. The estimate~\eqref{eq:localerr} immediately follows from the $\mb{P}_2(K)$-invariance of the local interpolation operator $\pi_K$~\cite{CiarletRaviart:1972}.
\end{proof}

\subsection{The mixed finite element approximation}
We are ready to construct a stable finite element pair to approximate~\eqref{eq:mix}. Define
\[
X_h{:}=\set{v\in H^1(\Om)}{v|_K\in P_K\quad\forall K\in\T_h(\Omega), \int_e\jump{\partial_{\boldsymbol{n}}v}\dsx=0\quad\forall e\in\mc{E}_h^i},
\]
and the corresponding homogeneous space is defined as $X_h^0=X_h\cap H_0^1(\Omega)$. Let $V_h=[X_h^0]^2\subset[H_0^1(\Omega)]^2$. %The definition of $P$ suggests that $p$ tends to $0$ near the vertices of $\Omega$. 
Motivated by $P$, we define
\begin{align*}
P_h{:}=\Bigl\{q\in L_0^2(\Omega)\cap H^1(\Omega)\ | \ &q|_K\in\mb{P}_1(K)\text{ for all } K\in\mc{T}_h(\Omega),\\
&q(\boldsymbol{a})=0\text{ for all }\boldsymbol{a}\in\mc{V}(\Omega)\Bigr\}.
\end{align*}

Theorem~\ref{thm:Hardy} indicates that $P_h\subset P$ for a fixed $h$. This observation motivates the following lemma, which is key to prove the discrete inf-sup condition.

\begin{lemma}\label{lema:divPh}
For any $q\in P_h$, there exists $\boldsymbol{v}\in V$ such that $\divop\boldsymbol{v}=q$ and
\[
\nm{\boldsymbol{v}}{H^1}\le C\nm{q}{L^2},\quad\nm{\boldsymbol{v}}{H^2}\le C\log(1/h)\nm{\nabla q}{L^2}.\]
\end{lemma}

\begin{proof}
For any $q\in P_h$, it follows from  Lemma~\ref{lema:divP} that there exists $v\in V$ satisfying $\divop\boldsymbol{v}=q$, and there exists $C$ such that
\[
\nm{\boldsymbol{v}}{H^1}\le C\nm{q}{L^2},
\qquad\text{and\quad}\nm{\boldsymbol{v}}{H^2}\le C\nm{q}{H^1_+}.
\]
Using the broken Hardy inequality in Theorem~\ref{thm:Hardy}, 
we obtain
\[
\nm{q}{H^1_+}\le C\log(1/h)\nm{q}{H^1}
\le C\log(1/h)\nm{\nabla q}{L^2},
\]
where we have used the Poincare's inequality in the last step because $q$ has a vanishing mean over $\Om$. A combination of the above two inequalities completes the proof.
\end{proof}%

The mixed finite element approximation problem reads as: Find $(\boldsymbol{u}_h,p_h)\in V_h\x P_h$ such that
\begin{equation}\label{eq:mixap}
\left\{\begin{aligned}
a_{\iota,h}(\boldsymbol{u}_h,v)+b_{\iota,h}(\boldsymbol{v},p_h)&=(\boldsymbol{f},\boldsymbol{v})\qquad&&\text{for all\quad}\boldsymbol{v}\in V_h,\\
b_{\iota,h}(\boldsymbol{u}_h,q)-\lam^{-1}c_\iota(p_h,q)&=0\qquad&&\text{for all\quad}q\in P_h,
\end{aligned}\right.
\end{equation}
where for all $\boldsymbol{v},\boldsymbol{w}\in V_h$,
\[
a_{\iota,h}(\boldsymbol{v},\boldsymbol{w}){:}=2\mu\lr{(\boldsymbol{\epsilon}(\boldsymbol{v}),\boldsymbol{\epsilon}(\boldsymbol{w}))+\iota^2(\na_h\boldsymbol{\epsilon}(\boldsymbol{v}),\na_h\boldsymbol{\epsilon}(\boldsymbol{w}))},
\]
and for all $\boldsymbol{v}\in V_h$ and $q\in P_h$,
\[
b_{\iota,h}(\boldsymbol{v},q){:}=(\divop\boldsymbol{v}, q)+\iota^2(\na_h\divop\boldsymbol{v},\na q).
\]
For all $\boldsymbol{v}\in V_h$, we define a broken norm
\(
\nm{\na\boldsymbol v}{\iota,h}{:}=\nm{\na\boldsymbol{v}}{L^2}+\iota\nm{\na_h^2\boldsymbol{v}}{L^2}.
\)

The following broken Korn's inequality is proved in~\cite[Theorem 2]{LiMingWang:2021}: For all $\boldsymbol{v}\in V_h$, there holds
\[
  \nm{\na_h\boldsymbol{\epsilon}(\boldsymbol{v})}{L^2}\ge\Lr{1-1/\sqrt{2}}\nm{\na_h^2\boldsymbol{v}}{L^2},
\]
which together with the first Korn's inequality~\eqref{eq:1stkorn} gives
\begin{equation}\label{eq:disell}
a_{\iota,h}(\boldsymbol{v},\boldsymbol{v})\ge\dfrac\mu2\nm{\nabla\boldsymbol{v}}{\iota,h}^2\qquad\text{for all }\boldsymbol{v}\in V_h.
\end{equation}

It remains to prove the discrete inf-sup condition for the pair $(V_h,P_h)$. A natural way is to construct a uniformly stable quasi-Fortin operator. We achieve this by constructing different Fortin operators for the different regimes of $\iota/h$.

Firstly we define a natural interpolation operator $\varPi_h:V\to V_h$ by $\varPi_h|_K=\varPi_K=[\pi_K]^2$. It is straightforward to verify
\begin{lemma}%\label{lema:fortin1}
For all $\boldsymbol{v}\in V$, there holds
\begin{equation}\label{eq:inter2}
b_{\iota,h}(\varPi_h\boldsymbol{v},p)=b_\iota(\boldsymbol{v},p)\qquad\text{for all\;}p\in P_h.
\end{equation}
\end{lemma}

\begin{proof}
Using the fact that $\varPi_h\boldsymbol{v}\in V_h\subset [H_0^1(\Omega)]^2$, an integration by parts gives
\begin{equation}\label{eq:dividen1}
\int_{\Omega}\divop(\boldsymbol{v}-\varPi_h\boldsymbol{v})p\dx
=-\sum_{K\in\mathcal{T}_h}\int_K(\boldsymbol{v}-\varPi_K\boldsymbol{v})\cdot\nabla p\dx=0,
\end{equation}
where we have used the identity~\eqref{eq:inter1}$_4$ in the last step.

Next, an integration by parts yields
\[
\int_{\Omega}\nabla\divop(\boldsymbol{v}-\varPi_h\boldsymbol{v})\cdot\nabla p\dx
=\sum_{K\in\mathcal{T}_h}\int_{\partial K}\divop(\boldsymbol{v}-\varPi_K\boldsymbol{v})\partial_{\boldsymbol{n}} p\dsx=0.
\]
The term vanishes because of $\partial_j=t_j\partial_{\boldsymbol{t}}+n_j\partial_{\boldsymbol{n}}$ for components $j=1,2$, and the relations~\eqref{eq:inter1}$_1$ and~\eqref{eq:inter1}$_3$.
\end{proof}

Next we introduce a regularized interpolation operator.
\begin{lemma}
There exists an operator $I_h:V\to V_h$ satisfying
\begin{equation}\label{eq:dividen2}
\int_{\Om}\divop(\boldsymbol{v}-I_h\boldsymbol{v})p\dx=0\qquad\text{for all\quad} p\in P_h.
\end{equation}

Moreover, if $\boldsymbol{v}\in V\cap[H^k(\Omega)]^2$ with an integer $k$ satisfying $1\le k\le 3$, then
\begin{equation}\label{eq:globalerr}
\nm{\na_h^j(\boldsymbol{v}-I_h\boldsymbol{v})}{L^2}\le Ch^{k-j}\nm{\na^k\boldsymbol{v}}{L^2}\qquad\text{for\quad}0\le j\le k.
\end{equation}
\end{lemma}

\begin{proof}
Let
\[
M_h{:}=\set{v\in H_0^1(\Omega)}{v|_K\in \mb{P}_2(K)+\text{span}\{b_K\}\text{ for all }K\in\T_h(\Om)}
\]
be the second-order MINI element~\cite{Arnold:1984}. There exists an interpolation operator $\varPi_M:H_0^1(\Omega)\cap H^k(\Omega)\to M_h$ satisfying
\begin{equation}\label{eq:SCerr}
\nm{\na_h^j(v-\varPi_Mv)}{L^2}\le Ch^{k-j}\nm{\na^kv}{L^2},\quad 0\le j\le k,1\le k\le 3,
\end{equation}
and the identity
\begin{equation}\label{eq:dividen3}
  \int_K\varPi_Mv\dx=\int_Kv\dx,\quad K\in\T_h.
\end{equation}
Define a global interpolation operator $\varPi_0:M_h\to X_h$ as
\begin{equation}\label{eq:inter3}
\left\{\begin{aligned}
\varPi_0 w(\boldsymbol{a})&=w(\boldsymbol{a})\quad&&\text{for all vertices }\boldsymbol{a},\\
\varPi_0 w(\boldsymbol{b})&=w(\boldsymbol{b})\quad&&\text{for all mid-points } \boldsymbol{b}\text{ of each edges},\\
\int_e\partial_{\boldsymbol{n}}\varPi_0 w\md\sigma(\boldsymbol{x})&=\int_e\ave{\partial_{\boldsymbol{n}}w}\md\sigma(\boldsymbol{x})\quad&&\text{for all edges }e\in\mc{E}_h,\\
\int_K\varPi_0w\dx&=\int_Kw\dx\quad&&\text{for all triangles }K\in\T_h(\Om).
\end{aligned}\right.
\end{equation}

Proceeding along the same line that leads to~\cite[Lemma 2]{GuzmanLeykekhmanNeilan:2012}, for any $w\in M_h$ and $K\in\T_h(\Om)$,
\[\nm{w-\varPi_0w}{L^2(K)}\le Ch_K^{3/2}\sum_{e\in\pa K\cap\mc{E}_h^i}\nm{\jump{\partial_{\boldsymbol{n}}w}}{L^2(e)}.
\]

Define $I_h{:}=[\varPi_0\varPi_M]^2:V\cap[H^k(\Omega)]^2\to V_h$. The estimate~\eqref{eq:globalerr} follows from the above inequality and the interpolation properties~\eqref{eq:SCerr} for $\varPi_M$. Similar to~\eqref{eq:dividen1}, using the identities~\eqref{eq:dividen3} and~\eqref{eq:inter3}$_4$, an integration by parts gives
\[
\int_{\Omega}\divop(\boldsymbol{v}-I_h\boldsymbol{v})p\dx=-\sum_{K\in\mathcal{T}_h}\int_K(\boldsymbol{v}-I_h\boldsymbol{v})\cdot\nabla p\dx=0.
\]
\end{proof}

We are ready to prove the discrete inf-sup condition, which is uniform in $\iota$.
\begin{lemma}\label{lema:infsup}
There exists $\beta$ independent of $\iota$ and $h$ such that 
\begin{equation}\label{eq:disbb}
\sup_{\boldsymbol{v}\in V_h}\dfrac{b_{\iota,h}(\boldsymbol{v},p)}{\nm{\nabla\boldsymbol{v}}{\iota,h}}\ge\dfrac{\beta}{\log^{3/2}(1/h)}\nm{p}{\iota}\qquad\text{for all\quad}p\in P_h.
\end{equation}
\end{lemma}

\begin{proof}
Using Lemma~\ref{lema:divPh}, for any $p\in P_h$, there exists $\boldsymbol{v}_0\in V$ such that $\divop\boldsymbol{v}_0=p$. Hence 
\[
b_\iota(\boldsymbol{v}_0,p)=\nm{p}{\iota}^2\qquad\text{and}\qquad\nm{\nabla\boldsymbol{v}_0}{\iota}\le C\log(1/h)\nm{p}{\iota}.
\]

We firstly consider the case $(\iota/h)\log^{1/2}(1/h)\le \gamma$ with $\gamma$ to be determined later on. Using~\eqref{eq:globalerr} with $j=k=2$, we obtain 
\[
\nm{\na_h\divop(\boldsymbol{v}_0-I_h\boldsymbol{v}_0)}{L^2}\le C\nm{\na^2\boldsymbol{v}_0}{L^2}\le C\log(1/h)\nm{\na p}{L^2},
\]
where we have used Lemma~\ref{lema:divPh} in the last step. This inequality together with the inverse inequality leads to
\begin{align*}
\iota^2\abs{(\na_h\divop(\boldsymbol{v}_0-I_h\boldsymbol{v}_0),\na p)}&\le C\iota^2\log(1/h)\nm{\na p}{L^2}^2\le C_{\ast}(\iota/h)^2\log(1/h)\nm{p}{L^2}^2\\
&\le \gamma^2C_{\ast}\nm{p}{\iota}^2.
\end{align*}
Fixing $\gamma$ such that $\gamma^2C_{\ast}<1$, we obtain
\[
b_{\iota,h}(I_h\boldsymbol{v}_0,p)=b_{\iota}(\boldsymbol{v}_0,p)-\iota^2(\na_h\divop(\boldsymbol{v}_0-I_h\boldsymbol{v}_0),\na p)\ge (1-\gamma^2C_{\ast})\nm{p}{\iota}^2.
\]
By~\eqref{eq:globalerr} and Lemma~\ref{lema:divPh}, we obtain
\[
\nm{\na I_h\boldsymbol{v}_0}{\iota,h}\le\nm{\na \boldsymbol{v}_0}{\iota}+\nm{\na(\boldsymbol{v}_0-I_h\boldsymbol{v}_0)}{\iota,h}\le C\nm{\na \boldsymbol{v}_0}{\iota}\le C\log(1/h)\nm{p}{\iota},
\]
A combination of the above two inequalities leads to
\begin{equation}\label{eq:disbb1}
\sup_{\boldsymbol{v}\in V_h}\dfrac{b_{\iota,h}(\boldsymbol{v},p)}{\nm{\nabla\boldsymbol{v}}{\iota,h}}\ge\dfrac{b_{\iota,h}(I_h\boldsymbol{v}_0,p)}{\nm{\na I_h\boldsymbol{v}_0}{\iota,h}}\ge\dfrac{1-\gamma^2C_*}{C\log(1/h)}\nm{p}{\iota}.
\end{equation}

Next, if $(\iota/h)\log^{1/2}(1/h)>\gamma$, then we use~\eqref{eq:localerr} and obtain
\[
\nm{\na(\boldsymbol{v}-\varPi_h\boldsymbol{v})}{L^2}\le Ch\nm{\nabla^2\boldsymbol{v}}{L^2}\quad\text{and}\quad
\nm{\na^2_h(\boldsymbol{v}-\varPi_h\boldsymbol{v})}{L^2}\le C\nm{\nabla^2\boldsymbol{v}}{L^2}.
\]
Therefore,
\begin{align*}
\nm{\na\varPi_hv}{\iota,h}&\le\nm{\nabla\boldsymbol{v}}{\iota}+\nm{\na(\boldsymbol{v}-\varPi_h\boldsymbol{v})}{\iota,h}\le\nm{\nabla\boldsymbol{v}}{\iota}+C(h+\iota)\nm{\nabla^2\boldsymbol{v}}{L^2}\\
&\le C (2+\log^{1/2}(1/h)/\gamma)\nm{\nabla\boldsymbol{v}}{\iota},
\end{align*}
which together with~\eqref{eq:inter2} gives 
\begin{equation}\label{eq:disbb2}
\begin{aligned}
\sup_{\boldsymbol{v}\in V_h}\dfrac{b_{\iota,h}(\boldsymbol{v},p)}{\nm{\nabla\boldsymbol{v}}{\iota}}&\ge\dfrac{b_{\iota,h}(\varPi_h\boldsymbol{v}_0,p)}{\nm{\na\varPi_h\boldsymbol{v}_0}{\iota,h}}
=\dfrac{b_{\iota,h}(\boldsymbol{v}_0,p)}{\nm{\na\varPi_h\boldsymbol{v}_0}{\iota,h}}=\dfrac{\nm{p}{\iota}^2}{\nm{\na\varPi_h\boldsymbol{v}_0}{\iota,h}}\\
&\ge\dfrac{\gamma}{C\log(1/h)(\log^{1/2}(1/h)+2\gamma)}\nm{p}{\iota}.
\end{aligned}
\end{equation}

A combination of~\eqref{eq:disbb1} and~\eqref{eq:disbb2} shows that~\eqref{eq:disbb} holds true with $\beta$ independent of $\iota$ and $h$.
\end{proof}

\subsection{Error estimates}
For any $\boldsymbol{v},\boldsymbol{w}\in V_h, q,z\in P_h$, we define the bilinear form 
\[
\mc{A}_{\iota,h}(\boldsymbol{v},q;\boldsymbol{w},z){:}=a_{\iota,h}(\boldsymbol{v},\boldsymbol{w})+b_{\iota,h}(\boldsymbol{w},q)+b_{\iota,h}(\boldsymbol{v},z)-\lam^{-1}c_{\iota}(q,z),
\]
and the norm $\wnm{(\boldsymbol{w},z)}{:}=\nm{\na\boldsymbol{w}}{\iota,h}+\nm{q}{\iota}+\lam^{-1/2}\nm{q}{\iota}$. 

\begin{lemma}
    There exists $\al$ depending on $\mu$ and $\beta$ such that
\begin{equation}\label{eq:infsup}
\inf_{(\boldsymbol{v},q)\in V_h\x P_h}\sup_{(\boldsymbol{w},z)\in V_h\x P_h}\dfrac{\mc{A}_{\iota,h}(\boldsymbol{v},q;\boldsymbol{w},z)}{\wnm{(\boldsymbol{w},z)}\wnm{(\boldsymbol{v},q)}}\ge\dfrac{\al}{\log^3(1/h)}.
\end{equation}
\end{lemma}

\begin{proof}
By the discrete inf-sup condition~\eqref{eq:disbb}, for any $q\in P_h$, there exists $\boldsymbol{v}_1\in V_h$ such that
\[
b_{\iota,h}(\boldsymbol{v}_1,q)=\nm{q}{\iota}^2,\qquad\nm{\nabla\boldsymbol{v}_1}{\iota,h}\le\dfrac{\log^{3/2}(1/h)}{\beta}\nm{q}{\iota}.
\]

For any $\del>0$ to be determined later on, and for any $\boldsymbol{v}\in V_h$ and $q\in P_h$, we take $\boldsymbol{w}=\boldsymbol{v}+\del \boldsymbol{v}_1$ and $z=-q$, then
\begin{align*}
\mc{A}_{\iota,h}(\boldsymbol{v},q;\boldsymbol{w},z)&=\mc{A}_{\iota,h}(\boldsymbol{v},q;\boldsymbol{v},-q)+\del\mc{A}_{\iota,h}(\boldsymbol{v},q;\boldsymbol{v}_1,0)\\
&\ge\dfrac{\mu}{2}\nm{\nabla\boldsymbol{v}}{\iota,h}^2+\lam^{-1}\nm{q}{\iota}^2+\del\nm{q}{\iota}^2-\dfrac{2\mu\del}{\beta}\log^{3/2}(1/h)\nm{\nabla\boldsymbol{v}}{\iota,h}\nm{q}{\iota}\\
&\ge\Lr{\dfrac{\mu}{2}-\dfrac{2\mu^2\del}{\beta^2}\log^3(1/h)}\nm{\nabla\boldsymbol{v}}{\iota,h}^2+\dfrac\del{2}\nm{q}{\iota}^2+\lam^{-1}\nm{q}{\iota}^2.
\end{align*}
Choosing $\del=\mu\beta^2/(4\mu^2\log^3(1/h)+\beta^2)$, and a direct calculation gives
\[
\wnm{(\boldsymbol{w},z)}\le \Lr{1+\dfrac{\mu\beta\log^{3/2}(1/h)}{4\mu^2\log^3(1/h)+\beta^2}}\wnm{(\boldsymbol{v},q)}.
\]
A combination of the above two inequalities leads to~\eqref{eq:infsup}.
\end{proof}

The error estimate for the smooth solution may be found in

\begin{theorem}\label{thm:smooth}
If $\boldsymbol{u}\in [H^3(\Om)]^2$ and $p\in H^2(\Om)$, then
\begin{equation}\label{eq:finalerr1}
\wnm{(\boldsymbol{u}-\boldsymbol{u}_h,p-p_h)}\le C\log^3(1/h)(h^2+\iota h)\Lr{\nm{\boldsymbol{u}}{H^3}+\nm{p}{H^2}}.
\end{equation}
\end{theorem}

\begin{proof}
Let $\boldsymbol{v}=\boldsymbol{u}_h-\boldsymbol{u}_I$ and $q=p_h-p_I$ with $\boldsymbol{u}_I\in V_h$ the interpolation of $\boldsymbol{u}$ and $p_I\in P_h$ the first-order Lagrangian interpolation of $p$, respectively. 
\begin{align*}
\mc{A}_{\iota,h}(\boldsymbol{v},q;\boldsymbol{w},z)&=\mc{A}_{\iota,h}(\boldsymbol{u}_h,p_h;\boldsymbol{w},z)-\mc{A}_{\iota,h}(\boldsymbol{u},p;\boldsymbol{w},z)+\mc{A}_{\iota,h}(\boldsymbol{u}-\boldsymbol{u}_I,p-p_I;\boldsymbol{w},z)\\
&=(\boldsymbol{f},\boldsymbol{w})-\mc{A}(\boldsymbol{u},p;\boldsymbol{w},z)+\mc{A}_{\iota,h}(\boldsymbol{u}-\boldsymbol{u}_I,p-p_I;\boldsymbol{w},z)\\
&=\mc{A}_{\iota,h}(\boldsymbol{u}-\boldsymbol{u}_I,p-p_I;\boldsymbol{w},z)-\iota^2\sum_{e\in\mc{E}_h^i}\int_e\partial_{\boldsymbol{n}}\boldsymbol{\sigma n}\cdot\jump{\partial_{\boldsymbol{n}}\boldsymbol{w}}\dsx.
\end{align*}

The boundedness of $\mc{A}_{\iota,h}$ yields
\[
\abs{\mc{A}_{\iota,h}(\boldsymbol{u}-\boldsymbol{u}_I,p-p_I;\boldsymbol{w},z)}
\le C\wnm{(\boldsymbol{u}-\boldsymbol{u}_I,p-p_I)}\wnm{(\boldsymbol{w},z)},
\]
and the standard interpolation error estimates in~\eqref{eq:localerr} gives
\[
\wnm{(\boldsymbol{u}-\boldsymbol{u}_I,p-p_I)}\le C(h^2+\iota h)(\nm{\boldsymbol{u}}{H^3}+\nm{p}{H^2}).
\]
Using the trace inequality, we obtain
\begin{align*}
\iota^2\labs{\sum_{e\in\mc{E}_h^i}\int_e\partial_{\boldsymbol{n}}\boldsymbol{\sigma n}\cdot\jump{\partial_{\boldsymbol{n}}\boldsymbol{w}}\dsx}&\le C\iota^2h(\nm{\boldsymbol{u}}{H^3}+\nm{p}{H^2})\nm{\na_h^2 w}{L^2}\\
&\le C\iota h(\nm{\boldsymbol{u}}{H^3}+\nm{p}{H^2})\nm{\na w}{\iota,h}.
\end{align*}
A combination of the above three inequalities, the discrete inf-sup condition~\eqref{eq:infsup} and the triangle inequalities gives~\eqref{eq:finalerr1}. \end{proof}

%The error estimate above converges if the solution is smooth enough. 
Next we give the error estimate of $\wnm{(\boldsymbol{u}_0-\boldsymbol{u}_h,p_0-p_h)}$ under the regularity assumption that $f\in L^2(\Om)$, in order to explain the results for the solution with boundary layers.

\begin{theorem}\label{thm:layer}
Let $\boldsymbol{u}_0$ be the solution of~\eqref{eq:elas} and $p_0=\lam\divop\boldsymbol{u}_0$. There holds
\begin{equation}\label{eq:finalerr2}
\wnm{(\boldsymbol{u}_0-\boldsymbol{u}_h,p_0-p_h)}\le C\log^3(1/h)(\iota+h)\nm{\boldsymbol{f}}{L^2}.
\end{equation}
Moreover, if $\boldsymbol{u}_0\in [H^3(\Omega)]^2$ and $p_0\in H^2(\Omega)$, then
\begin{equation}\label{eq:finalerr3}
\wnm{(\boldsymbol{u}_0-\boldsymbol{u}_h,p_0-p_h)}\le C\log^3(1/h)(\iota+h^2)(\nm{\boldsymbol{u}_0}{H^3}+\nm{p_0}{H^2}).
\end{equation}
\end{theorem}

\begin{proof}
Let $\boldsymbol{v}=\boldsymbol{u}_h-\boldsymbol{u}_I^0$ and $q=p_h-p_I^0$,  with $\boldsymbol{u}_I^0\in V_h$ and $p_I^0\in P_h$ the interpolation of $\boldsymbol{u}_0$ and $p_0$, respectively. 
\begin{align*}
\mc{A}_{\iota,h}(\boldsymbol{v},q;\boldsymbol{w},z)&=\mc{A}_{\iota,h}(\boldsymbol{u}_h,p_h;\boldsymbol{w},z)-\mc{A}_{\iota,h}(\boldsymbol{u}_0,p_0;\boldsymbol{w},z)+\mc{A}_{\iota,h}(\boldsymbol{u}_0-\boldsymbol{u}_I^0,p_0-p_I^0;\boldsymbol{w},z)\\
&=(\boldsymbol{f},\boldsymbol{w})-\mc{A}_{\iota,h}(\boldsymbol{u}_0,p_0;\boldsymbol{w},z)+\mc{A}_{\iota,h}(\boldsymbol{u}_0-\boldsymbol{u}_I^0,p_0-p_I^0;\boldsymbol{w},z)\\
&=\mc{A}_{\iota,h}(\boldsymbol{u}_0-\boldsymbol{u}_I^0,p_0-p_I^0;\boldsymbol{w},z)-\iota^2(\D\na\boldsymbol{\epsilon}(\boldsymbol{u}_0),\na_h\boldsymbol{\epsilon}(\boldsymbol{w})).
\end{align*}
The boundedness of $\mc{A}_{\iota,h}$ gives
\[\mc{A}_{\iota,h}(\boldsymbol{u}_0-\boldsymbol{u}_I^0,p_0-p_I^0;\boldsymbol{w},z)\le C\wnm{(\boldsymbol{u}_0-\boldsymbol{u}_I^0,p_0-p_I^0)}\wnm{(\boldsymbol{w},z)}.\]
The standard interpolation error estimates in~\eqref{eq:localerr} 
\[\wnm{(\boldsymbol{u}_0-\boldsymbol{u}_I^0,p_0-p_I^0)}\le C(\iota h^k+h^{k+1})(\nm{\boldsymbol{u}_0}{H^{2+k}}+\nm{p_0}{H^{1+k}}),\qquad\text{for}\quad k=0,1.\]

The second term is bounded by
\[
\abs{\iota^2(\mb{D}\na\boldsymbol{\epsilon}(\boldsymbol{u}_0),\na_h\boldsymbol{\epsilon}(\boldsymbol{w}))}\le\iota(\nm{\boldsymbol{u}_0}{H^2}+\nm{p_0}{H^1})\nm{\na\boldsymbol{w}}{\iota,h}.
\]
A combination of the above inequalities, the discrete inf-sup condition~\eqref{eq:infsup} and the triangle inequalitieswith the regularity estimate~\eqref{eq:regelas} gives~\eqref{eq:finalerr2} and~\eqref{eq:finalerr3}.
\end{proof}

\begin{remark}
If we discretize $\boldsymbol{u}$ with the element in~\cite[~\S~4.1]{Huang:2023}, the logarithmic factor $\log^3(1/h)$ in Theorem~\ref{thm:smooth} and~\ref{thm:layer} may be improved to $\log^2(1/h)$.
\end{remark}
\section{Numerical examples}\label{sec:numer}
We present numerical examples to show the accuracy and the robustness of the proposed element. All examples are performed on the nonuniform mesh, which is generated by the function {\tt generateMesh} in the Partial Differential Equation Toolbox of MATLAB. We report the relative errors $\nm{\na(\boldsymbol{u}-\boldsymbol{u}_h)}{\iota,h}/\nm{\na\boldsymbol{u}}{\iota}$ and the rates of convergence.

In all the examples, $\Omega=(0,1)^2$ and $E=1$. The Lam\'e constants are given by
\[
\lambda=\dfrac{E\nu}{(1+\nu)(1-2\nu)}\quad\text{and} \quad \mu=\dfrac{E}{2(1+\nu)}.
\]
We are interested in the case when the Poisson's ratio $\nu$ is close to $0.5$. 

\subsection{The first example}
We first test the performance of the element for solving an incompressible problem. Let $\boldsymbol{u}=(u_1,u_2)$ with
\[
u_1=-\sin^4(\pi x)\sin^2(\pi y)\sin(2\pi y)\quad\text{and} \quad u_2=\sin^2(\pi x)\sin(2\pi x)\sin^4(\pi y).
\]
Note $\divop\boldsymbol{u}=0$, and $\boldsymbol{f}$ is independent of $\lambda$. We list the error and the rate of convergence in Table~\ref{tab:case1}. 
\begin{table}[htbp]\centering\caption{Relative errors and convergence rates for the 1st example}~\label{tab:case1}\begin{tabular}{lcccc}
\hline
$\iota\backslash h$ & 1/8 & 1/16 & 1/32 & 1/64\\
\hline
\multicolumn{5}{c}{$\nu=0.3000,\lambda=0.5769,\mu=0.3846$}\\
\hline
1e+00 & 2.683e-01 & 1.403e-01 & 6.797e-02 & 3.381e-02\\
rate & & 0.93 & 1.05 & 1.01\\
1e-06 & 4.551e-02 & 1.284e-02 & 3.049e-03 & 7.565e-04\\
rate & & 1.83 & 2.07 & 2.01\\
\hline
\multicolumn{5}{c}{$\nu=0.4999,\lambda=\text{1.6664e3}, \mu=0.3334$}\\
\hline
1e+00 & 2.681e-01 & 1.403e-01 & 6.796e-02 & 3.381e-02\\
rate & & 0.93 & 1.05 & 1.01\\
1e-06 & 4.552e-02 & 1.284e-02 & 3.049e-03 & 7.565e-04\\
rate & & 1.83 & 2.07 & 2.01\\
\hline
\end{tabular}\end{table}

In view of Table~\ref{tab:case1}, we observe that the method converges with first-order when $\iota$ is large, while converges with second-order when $\iota$ is small, which is in agreement with the error estimate in Theorem~\ref{thm:smooth}. Moreover, the convergence is uniform in $\lam$.
\subsection{The second example}
We test the performance of the element for solving a nearly incompressible problem with smooth solution. Let $\boldsymbol{u}=(u_1,u_2)$ and $p=\lam\divop\boldsymbol{u}$ with
\begin{align*}
u_1&=-2\sin^2(\pi x)\sin^2(\pi y)\sin(2\pi y),\quad u_2=\dfrac{2\mu+\lam}{\lam}\sin(2\pi x)\sin^4(\pi y).\\
p&=4\pi\mu\sin(2\pi x)\sin^2(\pi y)\sin(2\pi y).
\end{align*}
The solution satisfies the homogeneous boundary $\boldsymbol{u}|_{\pa\Om}=\partial_{\boldsymbol{n}}\boldsymbol{\sigma n}|_{\pa\Om}=0$. Moreover,%Let $\lambda\to\infty$,
\[\Bar{\boldsymbol{u}}{:}=\lim_{\lam\to\infty}\boldsymbol{u}=(-2\sin^2(\pi x)\sin^2(\pi y)\sin(2\pi y), \sin(2\pi x)\sin^4(\pi y))^T,\]
the problem tends to completely incompressible because $\divop\Bar{\boldsymbol{u}}=0$. While the condition on the boundary $\{0,1\}\times(0,1)$ is destroyed, 
\[
\partial_{\boldsymbol{n}}\Bar{\boldsymbol{\sigma}}\boldsymbol{n}|_{\{0,1\}\times(0,1)}=-8\mu\pi^2(\sin^2(\pi y)\sin(2\pi y), 0)^T
\]
with $\Bar{\boldsymbol{\sigma}}=\C\eps(\Bar{\boldsymbol{u}})=2\mu\eps(\Bar{\boldsymbol{u}})$.
\begin{table}[htbp]\centering\caption{Relative errors and convergence rates for the 2nd example}~\label{tab:case2}\begin{tabular}{lcccc}
\hline
$\iota\backslash h$ & 1/8 & 1/16 & 1/32 & 1/64\\
\hline
\multicolumn{5}{c}{$\nu=0.3000,\lambda=0.5769,\mu=0.3846$}\\
\hline
1e+00 & 2.411e-01 & 1.246e-01 & 6.038e-02 & 3.006e-02\\
rate & & 0.95 & 1.05 & 1.01\\
1e-06 & 3.499e-02 & 9.479e-03 & 2.244e-03 & 5.572e-04\\
rate & & 1.88 & 2.08 & 2.01\\
\hline
\multicolumn{5}{c}{$\nu=0.4999,\lambda=\text{1.6664e3}, \mu=0.3334$}\\
\hline
1e+00 & 2.585e-01 & 1.366e-01 & 6.598e-02 & 3.300e-02\\
rate & & 0.92 & 1.05 & 1.00\\
1e-06 & 4.150e-02 & 1.119e-02 & 2.618e-03 & 6.527e-04\\
rate & & 1.89 & 2.10 & 2.00\\
\hline
\end{tabular}\end{table}

The convergence rate is reported in Table~\ref{tab:case2}. The optimal order is also observed for the nearly incompressible problem, which is consistent with the theoretical prediction in Theorem~\ref{thm:smooth}.
\subsection{The third example}
This example is motivated by~\cite{Wihler:2006} and is used to mimic the singularity of the displacement for the crack tip~\cite{GaoPark:2007}. The exact solution $\boldsymbol{u}=(u_1,u_2)$ expressed in the polar coordinates reads as
\[
u_1=u_{\rho}(\rho,\theta)\cos\theta-u_{\theta}(\rho,\theta)\sin\theta,\quad  
u_2=u_{\rho}(\rho,\theta)\sin\theta+u_{\theta}(\rho,\theta)\cos\theta,
\]
where
\begin{align*}
u_{\rho}&=\dfrac1{2\mu}\rho^{\alpha}\lr{-(\alpha+1)\cos((\alpha+1)\theta)+(C_2-(\alpha+1))C_1\cos((\alpha-1)\theta)},\\
u_{\theta}&=\dfrac1{2\mu}\rho^{\alpha}\lr{(\alpha+1)\sin((\alpha+1)\theta)+(C_2+\alpha-1)C_1\sin((\alpha-1)\theta)},
\end{align*}
and $\alpha=1.5, \omega=3\pi/4$, 
\[
C_1=-\dfrac{\cos((\alpha+1)\omega)}{\cos((\alpha-1)\omega)}\quad\text{and}\quad C_2=\dfrac{2(\lambda+2\mu)}{\lambda+\mu}.
\]

The problem is nearly incompressible because
\[
\divop\boldsymbol{u}=-\dfrac{3(1+\sqrt2)}{\lambda+\mu}\rho^{1/2}\cos(\theta/2).
\]
A direct calculation gives that $f\equiv 0$, while the inhomogeneous boundary arise in~\eqref{eq:sgbvp}. It is straightforward to verify that $\boldsymbol{u}\in [H^{5/2-\varepsilon}(\Omega)]^2$ and $p\in H^{3/2-\varepsilon}(\Omega)$ for any small number $\varepsilon>0$. 
\begin{table}[htbp]\centering\caption{Relative errors and convergence rates for the 3rd example}~\label{tab:case4}\begin{tabular}{lcccc}
\hline
$\iota\backslash h$ & 1/8 & 1/16 & 1/32 & 1/64\\
\hline
\multicolumn{5}{c}{$\nu=0.3000,\lambda=0.5769,\mu=0.3846$}\\
\hline
1e+00 & 7.923e-02 & 5.652e-02 & 4.022e-02 & 2.859e-02\\
rate & & 0.49 & 0.49 & 0.49\\
1e-06 & 1.836e-03 & 6.576e-04 & 2.336e-04 & 8.279e-05\\
rate & & 1.48 & 1.49 & 1.50\\
\hline
\multicolumn{5}{c}{$\nu=0.4999,\lambda=\text{1.6664e3}, \mu=0.3334$}\\
\hline
1e+00 & 8.763e-02 & 6.268e-02 & 4.466e-02 & 3.179e-02\\
rate & & 0.48 & 0.49 & 0.49\\
1e-06 & 2.225e-03 & 7.970e-04 & 2.831e-04 & 1.003e-04\\
rate & & 1.48 & 1.49 & 1.50\\
\hline
\end{tabular}\end{table}

Table~\ref{tab:case4} shows that $\nm{\na(\boldsymbol{u}-\boldsymbol{u}_h)}{\iota,h}\simeq\mc{O}(h^{1/2})$ when $\iota$ is large and $\nm{\na(\boldsymbol{u}-\boldsymbol{u}_h)}{\iota,h}\simeq\mc{O}(h^{3/2})$ when $\iota$ is close to zero.
\subsection{The fourth example}
In the last example, we test the relative error and rate of convergence for the numerical solution to the unperturbed second-order elliptic system, i.e., we report the relative error $\nm{\na(\boldsymbol{u}_0-\boldsymbol{u}_h)}{\iota,h}/\nm{\na \boldsymbol{u}_0}{\iota}$ in Table~\ref{tab:case3}. When the microscope parameter $\iota\to 0$, the boundary value problem~\eqref{eq:sgbvp} changes to~\eqref{eq:elas}. Let $\boldsymbol{u}_0=(u_1^0,u_2^0)$ with
\[
u_1^0=-\sin^2(\pi x)\sin(2\pi y), \quad u_2^0=\sin(2\pi x)\sin^2(\pi y)
\]
be the solution of~\eqref{eq:elas}. The source term $f$ is computed by~\eqref{eq:elas}, which is independent of $\lambda$ because $\divop\boldsymbol{u}_0=0$. The exact solution $\boldsymbol{u}$ of~\eqref{eq:sgbvp} is unknown, while it is non-smooth and has a sharp layer near the boundary. 
\begin{table}[htbp]\centering\caption{Relative errors and convergence rates for the 4th example}~\label{tab:case3}\begin{tabular}{lcccc}
\hline
$\iota\backslash h$ & 1/8 & 1/16 & 1/32 & 1/64\\
\hline
\multicolumn{5}{c}{$\nu=0.3000,\lambda=0.5769,\mu=0.3846$}\\
\hline
1e-04 & 2.162e-02 & 5.775e-03 & 1.402e-03 & 3.455e-04\\
rate & & 1.90 & 2.04 & 2.02\\
1e-06 & 2.162e-02 & 5.775e-03 & 1.401e-03 & 3.451e-04\\
rate & & 1.90 & 2.04 & 2.02\\
\hline
\multicolumn{5}{c}{$\nu=0.4999,\lambda=\text{1.6664e3}, \mu=0.3334$}\\
\hline
1e-04 & 2.163e-02 & 5.775e-03 & 1.402e-03 & 3.455e-04\\
rate & & 1.90 & 2.04 & 2.02\\
1e-06 & 2.163e-02 & 5.775e-03 & 1.401e-03 & 3.451e-04\\
rate & & 1.90 & 2.04 & 2.02\\
\hline
\end{tabular}\end{table}

We observe that the convergence rate for the element is still optimal, which is consistent with Theorem~\ref{thm:layer}, hence the element is also robust for the unperturbed second-order elliptic system.
\section{Conclusion}
We establish a novel broken Hardy inequality within the critical Sobolev space $W^{1,d}$. As a direct application of this inequality, we develop a quasi-bounded inversion of the divergence operator, which is essential for proving the inf-sup stability of a finite element pair that approximates the strain gradient elasticity with natural boundary conditions over polygon. The method remains robust even in the incompressible limit, achieving nearly optimal rate of convergence uniformly with respect to the microscopic material parameter. While the finite element can be extended to three-dimensional problems, the inversion of the divergence operator may differ, which we intend to address in a future work.
\begin{appendix}
\section{Regularity estimates on the smooth domain}\label{appd:sge}
In this section, we assume that the domain $\Omega$ is smooth, and show that the regularity estimates for Problem~\eqref{eq:sgbvp} is uniform in $\lambda$.
\subsection{Lower regularity estimates}
Due to the following lemma, the pressure space $P$ reduces to $L_0^2(\Omega)\cap H^1(\Omega)$.
\begin{lemma}[{\cite[Theorem 3.1]{Danchin:2013}}]\label{lema:divop}
Let $\pa\Om\in C^2$. For $q\in P$, there exists $\boldsymbol{v}\in V$ such that $\divop\boldsymbol{z}=q$ and
\begin{equation}\label{eq:estdiv1}
\nm{\na\boldsymbol{v}}{L^2}\le C\nm{q}{L^2},\qquad\nm{\na^2\boldsymbol{v}}{L^2}\le C\nm{\na q}{L^2},
\end{equation}
where $C$ only depends on $\Omega$.
\end{lemma}

For all $\boldsymbol{v},\boldsymbol{w}\in V$ and $q,z\in P$, we define a bilinear form %the bilinear form $\mc{A}_\iota$ 
\[
\mc{A}_\iota(\boldsymbol{v},q;\boldsymbol{w},z){:}=a_\iota(\boldsymbol{v},\boldsymbol{w})+b_\iota(\boldsymbol{w},q)+b_\iota(\boldsymbol{v},z)-\lambda^{-1}c_\iota(q,z)
\]
and the norm $\inm{(\boldsymbol{w},z)}{:}=\nm{\na\boldsymbol{w}}{\iota}+\nm{z}{\iota}+\lambda^{-1/2}\nm{z}{\iota}$.

\begin{lemma}
There exists a unique solution $\boldsymbol{u}\in V$ and $p\in P$ satisfying~\eqref{eq:mix} such that
\begin{equation}\label{eq:estsolu}
\inm{(\boldsymbol{u},p)}\le C\nm{f}{H^{-1}}.
\end{equation}
\end{lemma}

\begin{proof}
Using Lemma~\ref{lema:korn}, we obtain
\begin{align*}
a_{\iota}(\boldsymbol{v},\boldsymbol{v})&=2\mu\Lr{\nm{\boldsymbol{\epsilon}(\boldsymbol{v})}{L^2}^2+\iota^2\nm{\na\boldsymbol{\epsilon}(\boldsymbol{v})}{L^2}^2}\\
&\ge 2\mu\Lr{\dfrac12\nm{\nabla\boldsymbol{v}}{L^2}^2+\iota^2\Lr{1-1/\sqrt2}\nm{\nabla^2\boldsymbol{v}}{L^2}^2}\\
&\ge\dfrac{\mu}2\nm{\nabla\boldsymbol{v}}{\iota}^2.
\end{align*}

Next, by Lemma~\ref{lema:divop}, for any $p\in P$, there exists $\boldsymbol{v}_0\in V$ such that $\divop\boldsymbol{v}_0=p$ and $\nm{\na \boldsymbol{v}_0}{\iota}\le C\nm{p}{\iota}$. Therefore,
\[
\sup_{\boldsymbol{v}\in V}\dfrac{b_\iota(\boldsymbol{v},p)}{\nm{\nabla\boldsymbol{v}}{\iota}}\ge\dfrac{b_{\iota}(\boldsymbol{v}_0,p)}{\nm{\na \boldsymbol{v}_0}{\iota}}=\dfrac{\nm{p}{\iota}^2}{\nm{\na \boldsymbol{v}_0}{\iota}}\ge C\nm{p}{\iota}.
\]

Given the ellipticity of $a_{\iota}$ over $V$ and the inf-sup condition for $b_{\iota}$, we use~\cite[Theorem 2]{Braess:1996} to conclude that
\begin{equation}\label{eq:coninfsup}
  \inf_{(\boldsymbol{v},q)\in V\times P}\sup_{(\boldsymbol{w},z)\in V\times P}\dfrac{\mc{A}_\iota(\boldsymbol{v},q;\boldsymbol{w},z)}{\inm{(\boldsymbol{w},z)}\inm{(\boldsymbol{v},q)}}\ge C.
\end{equation}
It immediately gives the well-posedness of~\eqref{eq:mix} and the estimate~\eqref{eq:estsolu} by noting
\[
  \abs{(\boldsymbol{f},\boldsymbol{v})}\le\nm{f}{H^{-1}}\nm{\boldsymbol{v}}{H^1}\le C\nm{f}{H^{-1}}\nm{\nabla\boldsymbol{v}}{L^2}\le C\nm{f}{H^{-1}}\nm{\nabla\boldsymbol{v}}{\iota}.
\]
\end{proof}

\begin{theorem}\label{thm:reg1}
 Let $(\boldsymbol{u},p)$ be the solution of~\eqref{eq:mix}, $\boldsymbol{u}_0$ be the solution of~\eqref{eq:elas}, and $p_0=\lam\divop\boldsymbol{u}_0$. There holds
 \begin{equation}\label{eq:est}
     \inm{(\boldsymbol{u}-\boldsymbol{u}_0,p-p_0)}\le C\iota\nm{\boldsymbol{f}}{L^2}.
 \end{equation}
 Moreover, if $\boldsymbol{u}_0\in[H^3(\Omega)]^2$ and $p_0\in H^2(\Omega)$, then
\begin{equation}\label{eq:esthigh}
\inm{(\boldsymbol{u}-\boldsymbol{u}_0,p-p_0)}\le C\iota^{3/2}(\nm{\boldsymbol{u}_0}{H^3}+\nm{p_0}{H^2}).
\end{equation}

\end{theorem}

\begin{proof}
For any $\boldsymbol{v}\in V$ and $q\in P$, a direct calculation gives
\[
  \mc{A}_\iota(\boldsymbol{u}-\boldsymbol{u}_0,p-p_0;\boldsymbol{v},q)=\mc{A}_0(\boldsymbol{u}_0,p_0;\boldsymbol{v},q)-\mc{A}_\iota(\boldsymbol{u}_0,p_0;\boldsymbol{v},q)=-\iota^2(\D\na\boldsymbol{\epsilon}(\boldsymbol{u}_0),\na\boldsymbol{\epsilon}(\boldsymbol{v})).
\]
By~\eqref{eq:regelas}, we bound
\[
\abs{\iota^2(\mb{D}\na\boldsymbol{\epsilon}(\boldsymbol{u}_0),\na\boldsymbol{\epsilon}(\boldsymbol{v}))}\le C\iota^2(\nm{\boldsymbol{u}_0}{H^2}+\nm{p_0}{H^1})\nm{\nabla^2\boldsymbol{v}}{L^2}\le C\iota\nm{\boldsymbol{f}}{L^2}\nm{\nabla\boldsymbol{v}}{\iota}.
\]
This together with~\eqref{eq:coninfsup} gives~\eqref{eq:est}.

Integration by parts, we obtain
\[
-(\iota^2\mb{D}\na\boldsymbol{\epsilon}(\boldsymbol{u}_0),\na\boldsymbol{\epsilon}(\boldsymbol{v}))=\iota^2(\Delta\C\boldsymbol{\epsilon}(\boldsymbol{u}_0),\nabla\boldsymbol{v})-\iota^2\int_{\pa\Om}(\pa_{\boldsymbol{n}}\C\boldsymbol{\epsilon}(\boldsymbol{u}_0)\boldsymbol{n})\cdot\pa_{\boldsymbol{n}}\boldsymbol{v}\dsx.
\]
Using the trace inequality, we bound
\begin{align}\label{eq:est3}
\abs{\iota^2(\mb{D}\na\boldsymbol{\epsilon}(\boldsymbol{u}_0),\na\boldsymbol{\epsilon}(\boldsymbol{v}))}&\le C\iota^2\nm{\C\boldsymbol{\epsilon}(\boldsymbol{u}_0)}{H^2}\Lr{\nm{\na\boldsymbol{v}}{L^2}+\nm{\na\boldsymbol{v}}{L^2}^{1/2}\nm{\na^2\boldsymbol{v}}{L^2}^{1/2}}\nn\\
&\le C\iota^2\nm{\C\boldsymbol{\epsilon}(\boldsymbol{u}_0)}{H^2}\Lr{\nm{\na\boldsymbol{v}}{L^2}+\iota^{-1/2}\nm{\na\boldsymbol{v}}{L^2}+\iota^{1/2}\nm{\na^2\boldsymbol{v}}{L^2}}\\
&\le C\iota^{3/2}\nm{\C\boldsymbol{\epsilon}(\boldsymbol{u}_0)}{H^2}\nm{\na\boldsymbol{v}}{\iota}.\nonumber
\end{align}
This together with~\eqref{eq:coninfsup} yields~\eqref{eq:esthigh}. 
\end{proof}

\subsection{Higher regularity estimates}
To step further, we study an auxiliary boundary value problem, which may be viewed as the leading order part of~\eqref{eq:sgbvp}:
\begin{equation}\label{eq:auxprob}
\Delta\mc{L}\boldsymbol{w}=\boldsymbol{F}\quad\text{in\quad}\Om,\qquad
\boldsymbol{w}=\partial_{\boldsymbol{n}}\C\boldsymbol{\epsilon}(\boldsymbol{w})\boldsymbol{n}=0\quad\text{on\quad}\pa\Om,
\end{equation}
where
\(\mc{L}\boldsymbol{w}=\mu\Delta\boldsymbol{w}+(\lam+\mu)\na\divop\boldsymbol{w}\).

The well-posedness of the above problem is as follows.%problem~\eqref{eq:auxprob}.
\begin{lemma}
There exists a unique $\boldsymbol{w}\in V$ satisfying~\eqref{eq:auxprob} such that
\begin{equation}\label{eq:auxpriori}
\nm{\boldsymbol{w}}{H^2}+\lam\nm{\divop\boldsymbol{w}}{H^1}\le C\nm{\boldsymbol{F}}{H^{-2}}.
\end{equation}
\end{lemma}

\begin{proof}
We write~\eqref{eq:auxprob} into a variational problem: Find $\boldsymbol{w}\in V$ such that
\[
A(\boldsymbol{w},\boldsymbol{v})=\int_{\Om}\boldsymbol{F}(x)\boldsymbol{v}(x)\dx\qquad\text{for all\quad}\boldsymbol{v}\in V,
\]
where for any $\boldsymbol{v},\boldsymbol{z}\in V$,  
\[
A(\boldsymbol{v},\boldsymbol{z}){:}=\int_{\Om}\Delta\mc{L}\boldsymbol{v}(x)\boldsymbol{z}(x)\dx=2\mu(\na\boldsymbol{\epsilon}(\boldsymbol{v}),\na\boldsymbol{\epsilon}(\boldsymbol{z}))+\lam(\na\divop\boldsymbol{v},\na\divop\boldsymbol{z}).
\]
For any $\boldsymbol{v}\in V$, using the H$^2$-Korn's inequality~\eqref{eq:h2korn} and the Poincar\'e inequality, we obtain
\[
A(\boldsymbol{v},\boldsymbol{v})\ge 2\mu\nm{\na\boldsymbol{\epsilon}(\boldsymbol{v})}{L^2}^2\ge\dfrac\mu2\nm{\na^2\boldsymbol{v}}{L^2}^2\ge C\nm{\boldsymbol{v}}{H^2}^2.
\]
The existence and uniqueness of $\boldsymbol{w}\in V$ follows from the Lax-Milgram theorem. Moreover, 
\begin{equation}\label{eq:auxpriori1}
\nm{\boldsymbol{w}}{H^2}\le C\nm{\boldsymbol{F}}{H^{-2}}.
\end{equation}
Noting that $\divop\boldsymbol{w}\in P$, using Lemma~\ref{lema:divop}, we obtain that, there exists $\boldsymbol{v}_0\in V$ such that 
\(
\divop\boldsymbol{v}_0=\divop\boldsymbol{w},
\)
and 
\(
\nm{\na^2 \boldsymbol{v}_0}{L^2}\le C\nm{\na\divop\boldsymbol{w}}{L^2}.
\)
Using~\eqref{eq:auxpriori1} and the above estimate, 
\begin{align*}
\lam\nm{\na\divop\boldsymbol{w}}{L^2}^2&=\lam(\na\divop\boldsymbol{w},\na\divop\boldsymbol{v}_0)=A(\boldsymbol{w},\boldsymbol{v}_0)-2\mu(\na\boldsymbol{\epsilon}(\boldsymbol{w}),\na\eps(\boldsymbol{v}_0))\\
&=\int_{\Om}\boldsymbol{F}(x)\boldsymbol{v}_0(x)\dx-2\mu(\na\boldsymbol{\epsilon}(\boldsymbol{w}),\na\eps(\boldsymbol{v}_0))\\
&\le\nm{\boldsymbol{F}}{H^{-2}}\nm{\boldsymbol{v}_0}{H^2}+2\mu\nm{\na^2\boldsymbol{w}}{L^2}\nm{\na^2\boldsymbol{v}_0}{L^2}\\
&\le C\nm{\boldsymbol{F}}{H^{-2}}\nm{\na\divop\boldsymbol{w}}{L^2}.
\end{align*}
This implies
\[
\lam\nm{\na\divop\boldsymbol{w}}{L^2}\le C\nm{\boldsymbol{F}}{H^{-2}}.
\]
Combining the above estimate and~\eqref{eq:auxpriori1} with the Poincar\'e inequality, we obtain~\eqref{eq:auxpriori}.
\end{proof}

 To prove a higher regularity estimate on $\boldsymbol{w}$, we follow~\cite{Vog:1983} by introducing auxiliary variables and transforming~\eqref{eq:auxprob} into an elliptic system. Let $\boldsymbol{z}=\mu\boldsymbol{w}$ and $q=(\lam+\mu)\divop\boldsymbol{w}$, we reshape~\eqref{eq:auxprob} into the following system: Find $\boldsymbol{z}\in V$ and $q\in P$ such that
\begin{equation}\label{eq:auxprob2}
\left\{\begin{aligned}
\Delta^2\boldsymbol{z}+\na\Delta q&=\boldsymbol{F}\qquad&&\text{in\quad}\Om,\\
\divop\boldsymbol{z}&=G\qquad&&\text{in\quad}\Om,\\
\boldsymbol{z}&=\boldsymbol{0}\qquad&&\text{on\quad}\pa\Om,\\
\partial_{\boldsymbol{n}}\Lr{2\boldsymbol{\epsilon}(\boldsymbol{z})+\dfrac{\lam}{\lam+\mu}q1_{2\x 2}}\boldsymbol{n}&=0\qquad&&\text{on\quad}\pa\Om,
\end{aligned}\right.
\end{equation}
in the sense of distribution, where $1_{2\x 2}$ is the $2$ by $2$ identity matrix, and $G=\mu\divop\boldsymbol{w}$.

Following we need to check that~\eqref{eq:auxprob2} is elliptic in the sense of ADN.
\begin{lemma}\label{lema:reg}
Assuming that $\Omega$ is smooth enough, $\lambda>\mu$ and $k$ is a nonnegative integer. If $\boldsymbol{z}\in [H^4(\Om)]^2$ and $q\in H^3(\Om)$ is the solution of~\eqref{eq:auxprob2} for $\boldsymbol{F}\in[H^k(\Om)]^2$ and $G\in H^{k+3}(\Om)$, then there exists $C$ independent of $\lam$ but depending on $\Om$ and $k$ such that
\begin{equation}\label{eq:reg1}
\nm{\boldsymbol{z}}{H^{k+4}}+\nm{q}{H^{k+3}}\le C\Lr{\nm{\boldsymbol{F}}{H^k}+\nm{G}{H^{k+3}}+\nm{\boldsymbol{z}}{L^2}+\nm{q}{L^2}}.
\end{equation}
\end{lemma}

\begin{proof}
For any $\boldsymbol{Q}\in\Om$ and $\boldsymbol{\Xi}=(\xi_1,\xi_2)\in\R^2$, the symbol matrix associated with the system~\eqref{eq:auxprob2} is
$L(\boldsymbol{Q},\boldsymbol{\Xi})=L(\boldsymbol{\Xi})$ with
\[
L(\boldsymbol{\Xi})=\begin{pmatrix}
\abs{\boldsymbol{\Xi}}^4&0&\xi_1\abs{\boldsymbol{\Xi}}^2\\
0&\abs{\boldsymbol{\Xi}}^4&\xi_2\abs{\boldsymbol{\Xi}}^2\\
\xi_1&\xi_2&0
\end{pmatrix}.
\]
We may assign the integers $s_1=s_2=0,s_3=-3$ and $t_1=t_2=4$ and $t_3=3$ such that
\(
\text{deg}L_{ij}(\boldsymbol{\Xi})=s_i+t_j.
\) The principal part of $L$ coincides with $L$ and \(
\abs{\det L}=\abs{\boldsymbol{\Xi}}^8.
\)
This ensures the uniform ellipticity of the system~\cite{Agmon:1964} with the ellipticity $1$.

Next we check the {\em supplementary condition} because we deal with a plane problem. Let $\boldsymbol{\Xi}$ and $\boldsymbol{\Xi}'$ be two nonzero vectors in $\R^2$, the system
\begin{equation}\label{eq:root}
\abs{\det L(\boldsymbol{\Xi}+\tau\boldsymbol{\Xi}')}=0
\end{equation}
has the roots
\[
\tau_0^{\pm}(\boldsymbol{\Xi},\boldsymbol{\Xi}')=\dfrac{-\boldsymbol{\Xi}\cdot\boldsymbol{\Xi}'\pm i\sqrt{\abs{\boldsymbol{\Xi}}^2\abs{\boldsymbol{\Xi}'}^2-(\boldsymbol{\Xi}\cdot\boldsymbol{\Xi}')^2}}{\abs{\boldsymbol{\Xi}'}^2},
\]
both roots have multiplicity $4$. If $\boldsymbol{\Xi}$ and $\boldsymbol{\Xi}'$ are linearly independent, then $\abs{\boldsymbol{\Xi}}\abs{\boldsymbol{\Xi}'}-\abs{\boldsymbol{\Xi}\cdot\boldsymbol{\Xi}'}>0$, and we
conclude that~\eqref{eq:root} has exactly 4 roots with positive imaginary part. This verifies the supplementary condition of~\cite[p. 39]{Agmon:1964}.

It remains to check that the boundary are {\em complementing}. The symbol matrix to any $\boldsymbol{Q}\in\pa\Om$ for the boundary of~\eqref{eq:auxprob2} is
\[
B(\boldsymbol{Q},\boldsymbol{\Xi})=B(\boldsymbol{\Xi})=\begin{pmatrix}
1 & 0 & 0\\
0 & 1 & 0\\
(n_1\xi_1+n\cdot\boldsymbol{\Xi})n\cdot\boldsymbol{\Xi}&n_2\xi_1n\cdot\boldsymbol{\Xi}&\dfrac{\lam}{\lam+\mu}n_1n\cdot\boldsymbol{\Xi}\\
n_1\xi_2 n\cdot\boldsymbol{\Xi}&(n_2\xi_2+n\cdot\boldsymbol{\Xi})n\cdot\boldsymbol{\Xi}&
\dfrac{\lam}{\lam+\mu}n_2 n\cdot\boldsymbol{\Xi}
\end{pmatrix},
\] 
where $\boldsymbol{n}=(n_1,n_2)$ is the unit outward normal depending on $\boldsymbol{Q}$\footnote{We cannot simply take a special outward normal for verifying the complementing boundary because~\eqref{eq:auxprob2} is not rotation invariant.}.

Choosing $r_1=r_2=-4$ and $r_3=r_4=-2$, we obtain
\[
\text{deg}B_{ij}(\boldsymbol{Q},\boldsymbol{\Xi})=r_i+t_j.
\]
It turns out that the principal part of $B$ coincides with $B$. 

The adjoint matrix of $L$ reads as
\[
\adj(\boldsymbol{\Xi})=\begin{pmatrix}
  -\xi_2^2\abs{\boldsymbol{\Xi}}^2 & \xi_1\xi_2\abs{\boldsymbol{\Xi}}^2 & -\xi_1\abs{\boldsymbol{\Xi}}^6\\
  \xi_1\xi_2\abs{\boldsymbol{\Xi}}^2 & -\xi_1^2\abs{\boldsymbol{\Xi}}^2 & -\xi_2\abs{\boldsymbol{\Xi}}^6\\
  -\xi_1\abs{\boldsymbol{\Xi}}^4 & -\xi_2\abs{\boldsymbol{\Xi}}^4 & \abs{\boldsymbol{\Xi}}^8
\end{pmatrix}.
\]

To check problem~\eqref{eq:auxprob2} satisfies the {\em complementing boundary}, we need to verify that the rows of the matrix $B(\boldsymbol{\Xi}+\tau\boldsymbol{n})\adj(\boldsymbol{\Xi}+\tau\boldsymbol{n})$ are linearly independent modulo $M^+(\boldsymbol{\Xi},\tau)$. Hence, we need to show that
\[
\sum_{i=1}^4C_iB_{ik}(\boldsymbol{\Xi}+\tau\boldsymbol{n})\adj_{kj}(\boldsymbol{\Xi}+\tau\boldsymbol{n})\equiv 0\qquad\mod M^+(\boldsymbol{\Xi},\tau)
\]
only if the constants $\{C_i\}_{i=1}^4$ are all $0$.

The entries of  $B(\boldsymbol{\Xi}+\tau\boldsymbol{n})\adj(\boldsymbol{\Xi}+\tau\boldsymbol{n})$ reads as
\[
  (B\adj)_{11}=-\abs{\boldsymbol{\Xi}}^2(\tau^2+\abs{\boldsymbol{\Xi}}^2),\quad (B\adj)_{12}=\tau\abs{\boldsymbol{\Xi}}(\tau^2+\abs{\boldsymbol{\Xi}}^2),\quad (B\adj)_{13}=-\tau(\tau^2+\abs{\boldsymbol{\Xi}}^2)^3,
\]
and
\[
  (B\adj)_{21}=\tau\abs{\boldsymbol{\Xi}}(\tau^2+\abs{\boldsymbol{\Xi}}^2),\quad (B\adj)_{22}=-\tau^2(\tau^2+\abs{\boldsymbol{\Xi}}^2),\quad (B\adj)_{23}=-\abs{\boldsymbol{\Xi}}(\tau^2+\abs{\boldsymbol{\Xi}}^2)^3,
\]
and
\begin{align*}
  (B\adj)_{31}&=-\Lr{2\abs{\boldsymbol{\Xi}}^2+\dfrac\lambda{\lambda+\mu}(\tau^2+\abs{\boldsymbol{\Xi}}^2)}\tau^2(\tau^2+\abs{\boldsymbol{\Xi}}^2),\\
 (B\adj)_{32}&=\Lr{2\tau^2-\dfrac\lambda{\lambda+\mu}(\tau^2+\abs{\boldsymbol{\Xi}}^2)}\tau\abs{\boldsymbol{\Xi}}(\tau^2+\abs{\boldsymbol{\Xi}}^2),\\
 (B\adj)_{33}&=-\Lr{2\tau^2-\dfrac\lambda{\lambda+\mu}(\tau^2+\abs{\boldsymbol{\Xi}}^2)}\tau(\tau^2+\abs{\boldsymbol{\Xi}}^2)^3,
\end{align*}
and
\begin{align*}
  (B\adj)_{41}&=\tau\abs{\boldsymbol{\Xi}}(\tau^2-\abs{\boldsymbol{\Xi}}^2)(\tau^2+\abs{\boldsymbol{\Xi}}^2),\quad(B\adj)_{42}=-\tau^2(\tau^2-\abs{\boldsymbol{\Xi}}^2)(\tau^2+\abs{\boldsymbol{\Xi}}^2),\\
(B\adj)_{43}&=-2\tau^2\abs{\boldsymbol{\Xi}}(\tau^2+\abs{\boldsymbol{\Xi}}^2)^3.
\end{align*}

We assume $\boldsymbol{n}=(n_1, n_2)$ and $\boldsymbol{\Xi}=\abs{\boldsymbol{\Xi}}(-n_2, n_1)$ with $\abs{\boldsymbol{\Xi}}\neq 0$. Therefore, $\boldsymbol{\Xi}+\tau\boldsymbol{n}=(\tau n_1-\abs{\boldsymbol{\Xi}} n_2,\tau n_2+\abs{\boldsymbol{\Xi}}n_1)$ and the equation $\det L(\boldsymbol{\Xi}+\tau\boldsymbol{n})=0$ has four roots with positive imaginary part, all of them equal to $i\abs{\boldsymbol{\Xi}}$, hence
\[
M^+(\boldsymbol{\Xi},\tau)=(\tau-i\abs{\boldsymbol{\Xi}})^4.
\]

The above conditions may be reshaped into: For $j=1,2,3$,
\begin{equation}\label{eq:complement}
f_j(\tau)\equiv 0\qquad\mod M^+(\boldsymbol{\Xi},\tau),
\end{equation}
where $f_j(\tau)=(\tau^2+\abs{\boldsymbol{\Xi}}^2)g_j(\tau)$ for $j=1,2$ and $f_3(\tau)=(\tau^2+\abs{\boldsymbol{\Xi}}^2)^3g_3(\tau)$ with
\begin{align*}
  g_1(\tau){:}&=-C_1(\tau n_2+\abs{\boldsymbol{\Xi}}n_1)^2+C_2(\tau n_1-\abs{\boldsymbol{\Xi}}n_2)(\tau n_2+\abs{\boldsymbol{\Xi}}n_1)\\
  &-C_3\tau\biggl(\abs{\boldsymbol{\Xi}}(\tau n_1-\abs{\boldsymbol{\Xi}}n_2)(\tau n_2+\abs{\boldsymbol{\Xi}}n_1)+\tau(\tau n_2+\abs{\boldsymbol{\Xi}}n_1)^2\\
  &\phantom{-C_3\tau\biggl(}\quad+\dfrac\lambda{\lambda+\mu}(\tau^2+\abs{\boldsymbol{\Xi}}^2)(\tau n_1-\abs{\boldsymbol{\Xi}} n_2) n_1\biggr)\\
  &+C_4\tau\biggl(\tau(\tau n_1-\abs{\boldsymbol{\Xi}}n_2)(\tau n_2+\abs{\boldsymbol{\Xi}}n_1)-\abs{\boldsymbol{\Xi}}(\tau n_2+\abs{\boldsymbol{\Xi}}n_1)^2\\
  &\phantom{-C_3\tau\biggl(}\quad-\dfrac\lambda{\lambda+\mu}(\tau^2+\abs{\boldsymbol{\Xi}}^2)(\tau n_1-\abs{\boldsymbol{\Xi}} n_2) n_2\biggr),
  %&=-C_1(\tau n_2+\abs{\boldsymbol{\Xi}}n_1)^2+C_2(\tau n_1-\abs{\boldsymbol{\Xi}}n_2)(\tau n_2+\abs{\boldsymbol{\Xi}}n_1)\\
  %&-C_3\Lr{\tau^4\Lr{1-\dfrac\mu{\lambda+\mu}n_1^2}+\tau^3\abs{\boldsymbol{\Xi}}\dfrac{2\lambda+3\mu}{\lambda+\mu} n_1 n_2-\tau^2\abs{\boldsymbol{\Xi}}^2\Lr{1-\dfrac{4\lambda+3\mu}{\lambda+\mu}n_1^2}-\tau\abs{\boldsymbol{\Xi}}^3\dfrac{2\lambda+\mu}{\lambda+\mu} n_1 n_2}\\
  %&+C_4\Lr{\tau^4\dfrac\mu{\lambda+\mu} n_1 n_2+\tau^3\abs{\boldsymbol{\Xi}}\Lr{1-\dfrac{2\lambda+3\mu}{\lambda+\mu}n_2}}\\
  %&-C_4\Lr{\tau^2\abs{\boldsymbol{\Xi}}^2\dfrac{4\lambda+3\mu}{\lambda+\mu} n_1 n_2+\tau\abs{\boldsymbol{\Xi}}^3\Lr{1-\dfrac{2\lambda+\mu}{\lambda+\mu}n_2}},
\end{align*}
and
\begin{align*}
  g_2(\tau){:}&=C_1(\tau n_1-\abs{\boldsymbol{\Xi}}n_2)(\tau n_2+\abs{\boldsymbol{\Xi}}n_1)-C_2(\tau n_1-\abs{\boldsymbol{\Xi}} n_2)^2\\
  &+C_3\tau\biggl(\tau(\tau n_1-\abs{\boldsymbol{\Xi}}n_2)(\tau n_2+\abs{\boldsymbol{\Xi}}n_1)+\abs{\boldsymbol{\Xi}}(\tau n_1-\abs{\boldsymbol{\Xi}} n_2)^2\\
  &\phantom{+C_3\tau\biggl(}\quad-\dfrac\lambda{\lambda+\mu}(\tau^2+\abs{\boldsymbol{\Xi}}^2)(\tau n_2+\abs{\boldsymbol{\Xi}}n_1) n_1\biggr)\\
  &+C_4\tau\biggl(\abs{\boldsymbol{\Xi}}(\tau n_1-\abs{\boldsymbol{\Xi}}n_2)(\tau n_2+\abs{\boldsymbol{\Xi}}n_1)-\tau(\tau n_1-\abs{\boldsymbol{\Xi}} n_2)^2\\
  &\phantom{+C_4\tau\biggl(}\quad-\dfrac\lambda{\lambda+\mu}(\tau^2+\abs{\boldsymbol{\Xi}}^2)(\tau n_2+\abs{\boldsymbol{\Xi}}n_1) n_2\biggr),
  %&=C_1(\tau n_1-\abs{\boldsymbol{\Xi}}n_2)(\tau n_2+\abs{\boldsymbol{\Xi}}n_1)-C_2(\tau n_1-\abs{\boldsymbol{\Xi}} n_2)^2\\
  %&+C_3\Lr{\tau^4\dfrac\mu{\lambda+\mu} n_1 n_2-\tau^3\abs{\boldsymbol{\Xi}}\Lr{1-\dfrac{2\lambda+3\mu}{\lambda+\mu}n_1^2}}\\
  %&-C_3\Lr{\tau^2\abs{\boldsymbol{\Xi}}^2\dfrac{4\lambda+3\mu}{\lambda+\mu} n_1 n_2-\tau\abs{\boldsymbol{\Xi}}^3\Lr{1-\dfrac{2\lambda+\mu}{\lambda+\mu}n_1^2}}\\
  %&-C_4\Lr{\tau^4\Lr{1-\dfrac\mu{\lambda+\mu}n_2}-\tau^3\abs{\boldsymbol{\Xi}}\dfrac{2\lambda+3\mu}{\lambda+\mu} n_1 n_2}\\
  %&+C_4\Lr{\tau^2\abs{\boldsymbol{\Xi}}^2\Lr{1-\dfrac{4\lambda+3\mu}{\lambda+\mu}n_2}-\tau\abs{\boldsymbol{\Xi}}^3\dfrac{2\lambda+\mu}{\lambda+\mu} n_1 n_2},
\end{align*}
and
\begin{align*}
  g_3(\tau){:}&=-C_1(\tau n_1-\abs{\boldsymbol{\Xi}} n_2)-C_2(\tau n_2+\abs{\boldsymbol{\Xi}}n_1)\\
  &-C_3\tau\Lr{2\tau(\tau n_1-\abs{\boldsymbol{\Xi}} n_2)-\dfrac\lambda{\lambda+\mu}(\tau^2+\abs{\boldsymbol{\Xi}}^2) n_1}\\
  &-C_4\tau\Lr{2\tau(\tau n_2+\abs{\boldsymbol{\Xi}}n_1)-\dfrac\lambda{\lambda+\mu}(\tau^2+\abs{\boldsymbol{\Xi}}^2) n_2}.
  %&=-C_1(\tau n_1-\abs{\boldsymbol{\Xi}} n_2)-C_2(\tau n_2+\abs{\boldsymbol{\Xi}}n_1)\\
  %&-C_3\Lr{\tau^3\dfrac{\lambda+2\mu}{\lambda+\mu} n_1-2\tau^2\abs{\boldsymbol{\Xi}} n_2-\tau\abs{\boldsymbol{\Xi}}^2\dfrac\lambda{\lambda+\mu} n_1}\\
  %&-C_4\Lr{\tau^3\dfrac{\lambda+2\mu}{\lambda+\mu} n_2+2\tau^2\abs{\boldsymbol{\Xi}} n_1-\tau\abs{\boldsymbol{\Xi}}^2\dfrac\lambda{\lambda+\mu} n_2}.
\end{align*}

Observing that $f_1(\tau)$ and $f_2(\tau)$ have only one root $i\abs{\boldsymbol{\Xi}}$, and $f_3(\tau)$ has root $i\abs{\boldsymbol{\Xi}}$ with multiplicity $3$, it follows from~\eqref{eq:complement} that %for $j=1,2$, there holds
\[
g_j(i\abs{\boldsymbol{\Xi}})=0,\quad g_j'(i\abs{\boldsymbol{\Xi}})=0,\quad
g_j^{\prime\prime}(i\abs{\boldsymbol{\Xi}})=0
\]
 for $j=1,2$ and $g_3(i\abs{\boldsymbol{\Xi}})=0$.

It follows from $g_1(i\abs{\boldsymbol{\Xi}})=g_2(i\abs{\boldsymbol{\Xi}})=g_3(i\abs{\boldsymbol{\Xi}})=0$ that
\begin{equation}\label{eq:rela0}
  C_1-i C_2-2\abs{\boldsymbol{\Xi}}^2C_3+2i\abs{\boldsymbol{\Xi}}^2C_4=0.
\end{equation}

Using $g_1^{\prime}(i\abs{\boldsymbol{\Xi}})=0$, we obtain
\begin{equation}\label{eq:rela1}\begin{aligned}
  2n_2C_1&-(n_1+i n_2)C_2-2\abs{\boldsymbol{\Xi}}^2\Lr{3n_2-\dfrac{\lambda+2\mu}{\lambda+\mu}i n_1}C_3\\
  &+2\abs{\boldsymbol{\Xi}}^2\Lr{2 n_1+\dfrac{2\lambda+3\mu}{\lambda+\mu}i n_2}C_4=0.
\end{aligned}\end{equation}

Using $g_2^\prime(i\abs{\boldsymbol{\Xi}})=0$, we obtain
\begin{equation}\label{eq:rela2}\begin{aligned}
  (n_1+i n_2)C_1&-2i n_1C_2-2\abs{\boldsymbol{\Xi}}^2\Lr{\dfrac{2\lambda+3\mu}{\lambda+\mu}n_1+2i n_2}C_3\\
  &-2\abs{\boldsymbol{\Xi}}^2\Lr{\dfrac{\lambda+2\mu}{\lambda+\mu}n_2-3i n_1}C_4=0.
\end{aligned}\end{equation}
Note the equation~\eqref{eq:rela2} is not independent of~\eqref{eq:rela0} and~\eqref{eq:rela1}.

Using $g_1^{\prime\prime}(i\abs{\boldsymbol{\Xi}})=0$, we obtain
\begin{equation}\label{eq:rela3}\begin{aligned}
  n_2^2C_1&-n_1n_2C_2-\abs{\boldsymbol{\Xi}}^2\Lr{\dfrac{3\lambda-2\mu}{\lambda+\mu}n_1^2+7n_2^2-\dfrac{6\lambda+9\mu}{\lambda+\mu}i n_1n_2}C_3\\
  &+\abs{\boldsymbol{\Xi}}^2\Lr{\dfrac{4\lambda+9\mu}{\lambda+\mu}n_1n_2-3i n_1^2+\dfrac{3\lambda+6\mu}{\lambda+\mu}i n_2^2}C_4=0.
\end{aligned}\end{equation}

Using $g_2^{\prime\prime}(i\abs{\boldsymbol{\Xi}})=0$, we obtain
\begin{equation}\label{eq:rela4}\begin{aligned}
  n_1n_2C_1&-n_1^2C_2-\abs{\boldsymbol{\Xi}}^2\Lr{\dfrac{4\lambda+9\mu}{\lambda+\mu}n_1n_2-\dfrac{3\lambda+6\mu}{\lambda+\mu}i n_1^2+3i n_2^2}C_3\\
  &+\abs{\boldsymbol{\Xi}}^2\Lr{7n_1^2+\dfrac{3\lambda-2\mu}{\lambda+\mu}n_2^2+\dfrac{6\lambda+9\mu}{\lambda+\mu}i n_1n_2}C_4=0.
\end{aligned}\end{equation}

We may write all the above equations~\eqref{eq:rela0},~\eqref{eq:rela1},~\eqref{eq:rela3},~\eqref{eq:rela4} into a compact form $A\boldsymbol{c}=0$ with $\boldsymbol{c}=(C_1,C_2,C_3,C_4)^T$. A direct calculation gives
\[
\det A=48\abs{\boldsymbol{\Xi}}^4\dfrac\lambda{\lambda+\mu}(n_1-i n_2)\neq 0.
\]

We conclude that
\[
C_1=C_2=C_3=C_4=0.
\]
Therefore, the boundary for the system~\eqref{eq:auxprob2} are {\em complementing}.

Finally, the estimate~\eqref{eq:reg1} follows from~\cite[Theorem 10.5]{Agmon:1964} and~\cite[Appendix D]{BochevGunzburg:2009}. The constant $C$ is independent of $\lam$ because the ellipticity constant, the roots of $L(\boldsymbol{\Xi}+\tau\boldsymbol{n})$, and the bound for the coefficients of operators $L$ and $B$ are independent of $\lam$.
\end{proof}

\begin{lemma}\label{lema:reg-aux}
Assume $\partial\Omega\in C^4$, let $\boldsymbol{w}$ be the solution of~\eqref{eq:auxprob}, then
\begin{equation}\label{eq:asspt}
\nm{\boldsymbol{w}}{H^3}+\lambda\nm{\divop\boldsymbol{w}}{H^2}\le C\nm{\boldsymbol{F}}{H^{-1}}.
\end{equation}
\end{lemma}

\begin{proof}
If we restrict $\lam$ to a bounded set, then~\eqref{eq:asspt} follows from the standard 
elliptic regularity estimate for Problem~\eqref{eq:auxprob}.

It suffices to consider Problem~\eqref{eq:auxprob} with large $\lam$. Without loss of generality, we assume that $\lam>\mu$. Using the standard elliptic regularity estimate, there exists a unique solution $\boldsymbol{w}\in [H^4(\Om)]^2$ when $\boldsymbol{F}\in[L^2(\Om)]^2$. Hence $\boldsymbol{z}=\mu \boldsymbol{w}\in [H^4(\Om)]^2$ and $q=(\lam+\mu)\divop\boldsymbol{w}\in H^3(\Om)$ are the solution of~\eqref{eq:auxprob2} with $G=\mu\divop\boldsymbol{w}$. Using~\eqref{eq:reg1} in terms of $\boldsymbol{z}=\mu \boldsymbol{w}$ and $q=(\lambda+\mu)\divop\boldsymbol{w}$ with $k=0$, we obtain
\[
\nm{\boldsymbol{w}}{H^4}+(\lam+\mu)\nm{\divop\boldsymbol{w}}{H^3}\le C\Lr{\nm{\boldsymbol{F}}{L^2}+\nm{\divop\boldsymbol{w}}{H^3}+\nm{\boldsymbol{w}}{H^1}+\lam\nm{\divop\boldsymbol{w}}{L^2}}.
\]
Using the a-priori estimate~\eqref{eq:auxpriori}, we rewrite the above inequality as
\[
\nm{\boldsymbol{w}}{H^4}+(\lam+\mu)\nm{\divop\boldsymbol{w}}{H^3}\le C_1\Lr{\nm{\boldsymbol{F}}{L^2}+\nm{\divop\boldsymbol{w}}{H^3}}.
\]
By
\[
\nm{\boldsymbol{w}}{H^4}+\lam\nm{\divop\boldsymbol{w}}{H^3}\le C_1\nm{\boldsymbol{F}}{L^2}+\dfrac{\lambda}{2}\nm{\divop\boldsymbol{w}}{H^3}
\]
when $\lam\ge 2C_1$. This gives
\[
\nm{\boldsymbol{w}}{H^4}+\lam\nm{\divop\boldsymbol{w}}{H^3}\le 2C_1\nm{\boldsymbol{F}}{L^2}.
\]

Interpolating between~\eqref{eq:auxpriori} and the above inequality, we obtain~\eqref{eq:asspt}. 
\end{proof}
We are ready to prove the higher regularity of the problem~\eqref{eq:mix} when $\Omega$ is smooth. This is the main result of this part.
\begin{theorem}\label{thm:reg2}
Let $(\boldsymbol{u},p)$ be the solution of~\eqref{eq:mix}, $\boldsymbol{u}_0$ be the solution of~\eqref{eq:elas}, and $p_0=\lam\divop\boldsymbol{u}_0$. There holds
\begin{equation}\label{eq:h3est}
\nm{\boldsymbol{u}}{H^3}+\nm{p}{H^2}\le C\iota^{-1}\nm{\boldsymbol{f}}{L^2}.
\end{equation}
Moreover if $\boldsymbol{u}_0\in[H^3(\Omega)]^2$ and $p_0\in H^2(\Omega)$, then
\begin{equation}\label{eq:h3esthigh}
\nm{\boldsymbol{u}}{H^3}+\nm{p}{H^2}\le C\iota^{-1/2}(\nm{\boldsymbol{u}_0}{H^3}+\nm{p_0}{H^2}).
\end{equation}
\end{theorem}

\begin{proof}
We rewrite~\eqref{eq:sgbvp} as
\[
\left\{\begin{aligned}
\mu\Delta^2u+(\mu+\lam)\Delta\na\divop\boldsymbol{u}&=\iota^{-2}\mc{L}(\boldsymbol{u}-\boldsymbol{u}_0)\qquad&&\text{in\quad}\Omega,\\
u=\partial_{\boldsymbol{n}}\boldsymbol{\sigma n}&=0\qquad&&\text{on\quad}\pa\Omega.
\end{aligned}\right.
\]
Using~\eqref{eq:asspt} and the error estimate, we obtain
\[
\nm{\boldsymbol{u}}{H^3}+\nm{p}{H^2}\le C\iota^{-2}\nm{\mc{L}(\boldsymbol{u}-\boldsymbol{u}_0)}{H^{-1}}\le C\iota^{-2}\inm{(\boldsymbol{u}-\boldsymbol{u}_0,p-p_0)},
\]
which together with Theorem~\ref{thm:reg1} gives~\eqref{eq:h3est} and~\eqref{eq:h3esthigh}.
\end{proof}
\end{appendix}
\bibliographystyle{amsplain}
\bibliography{sg-soft}
\end{document}